\documentclass[12pt]{amsart}%
\usepackage{a4wide,amsfonts,graphicx,
amssymb,stmaryrd,amscd,amsmath,latexsym,amsbsy,xypic}
\usepackage{marginnote}
\usepackage{etex}
\usepackage{amsfonts}
\usepackage{amsmath}
\usepackage[latin1]{inputenc}
\usepackage{setspace}
\usepackage{mathrsfs}
\usepackage{graphicx}
\usepackage{caption}
\usepackage{subcaption}
\usepackage{stmaryrd}
\usepackage[english]{babel}
\usepackage{enumitem}
\usepackage{amssymb}
\usepackage{amsthm}
\usepackage{bbding}
\usepackage[all]{xy}
\usepackage[a4paper]{geometry}
\usepackage{multido}
\providecommand{\U}[1]{\protect\rule{.1in}{.1in}}
\xyoption{all}
\graphicspath{ {./pics/} }
\numberwithin{equation}{section}
\theoremstyle{plain}
\newtheorem{thm}{Theorem}[section]

\newtheorem{prop}[thm]{Proposition}

\newtheorem{defi}[thm]{Definition}
\newtheorem{lem}[thm]{Lemma}
\newtheorem{cor}[thm]{Corollary}

\theoremstyle{remark}
\newtheorem{rema}[thm]{Remark}

\begin{document}
\title[Metaplectic Hecke algebra representations]{Metaplectic representations of Hecke algebras, Weyl group actions, and associated polynomials}
\author{Siddhartha Sahi, Jasper V. Stokman, Vidya Venkateswaran}

\address{S. Sahi, Department of Mathematics, Rutgers University, 110 Frelinghhuysen Rd,
Piscataway, NJ 08854-8019, USA.}
\email{sahi@math.rutgers.edu}
\address{J. Stokman, KdV Institute for Mathematics, University of Amsterdam, Science
Park 105-107, 1098 XG Amsterdam, The Netherlands.}
\email{j.v.stokman@uva.nl}
\address{V. Venkateswaran, Center for Communications Research, 805 Bunn Dr, Princeton,
NJ 08540, USA. }
\email{vvenkat@idaccr.org}

\begin{abstract}
Chinta and Gunnells introduced a rather intricate multi-parameter Weyl  group action on rational functions on a torus, which, when the parameters are
specialized to certain Gauss sums, describes the functional equations of Weyl group multiple Dirichlet series associated to metaplectic (n-fold) covers of algebraic
groups. In subsequent joint work with Puskas, they extended this action to a  ``metaplectic" representation of the equal parameter affine Hecke algebra, which allowed them to obtain explicit formulas for the p-parts of these Dirichlet series. They have also verified by a computer check the remarkable fact that their formulas 
continue to define a group action for general (unspecialized) parameters. 

In the first part of paper we give a conceptual explanation of this fact, by giving  a uniform and elementary construction of the ``metaplectic" representation for generic Hecke algebras as a suitable quotient of a parabolically induced affine Hecke algebra module, from which the associated Chinta-Gunnells Weyl group action follows through localization. 

In the second part of the paper we extend the metaplectic representation to the double affine Hecke algebra, which provides a generalization of Cherednik's basic representation. This allows us to introduce a new family of ``metaplectic" polynomials, which generalize nonsymmetric Macdonald polynomials. In this paper, we provide the details of the construction of metaplectic polynomials in type A; the general case will be handled in the sequel to this paper.
\end{abstract}
\maketitle

\vspace{2em}

%%%%%%%%%%%%%%%%%%%%%%%%%%%%%%%
%%%%%%%%%%%%%%%%%%%%%%%%%%%%%%%

\section{Introduction}

\vspace{2em}

This paper contains two main results concerning a somewhat mysterious action
of the Weyl group of a reductive Lie group on the algebra of rational
functions on its torus. This action was first introduced in type $A$ by
Kazhdan and Patterson \cite{KP}, and in full generality by Chinta and Gunnells
\cite{CG07, CG}, who used it to obtain formulas for the local parts
($p$-parts) of Weyl group multiple Dirichlet series. The action involves an
integer $n$ and parameters $g_{1},\ldots,g_{n-1}$, which in the application
are specialized to certain Gauss sums; however it remains a group action even
without this specialization. Chinta and Gunnells verified this fact through a
computer check and they asked for a conceptual proof. Our first main result
provides such a proof in complete generality.  The key role in the proof is
played by a certain representation of the affine Hecke algebra that we
construct in Theorem \ref{mainTHM} below, and which we refer to as the
\emph{metaplectic} representation. 

There is a striking analogy between the Chinta-Gunnells setting and the theory
of Macdonald polynomials \cite{Mac, Ma, Ch}. The latter are a family of orthogonal polynomials
on the torus that depend on two or three \textquotedblleft
root-length\textquotedblright\ parameters, and which generalize many important
polynomials in representation theory and algebraic combinatorics, including
spherical functions for real and $p$-adic groups. We show that there is much
more to this analogy. Our second main result is the construction of a family
of polynomials that we refer to as \emph{metaplectic} polynomials. These
depend on the root-length parameters as well as the $g_{1},\ldots,g_{n-1},$
and are a common generalization of nonsymmetric Macdonald polynomials \cite{Ma0,Ch0} and of the $p$-parts
of Weyl group multiple Dirichlet series.  A key point in our construction is
extending the metaplectic representation from the affine Hecke algebra to the
\emph{double} affine Hecke algebra.

In the present paper we discuss in detail the metaplectic polynomials only in
type A, where many of the essential ideas already appear. This is the setting
of \cite{KP} and of Macdonald's book on symmetric functions \cite{MacSym},
which is of considerable independent interest in algebraic combinatorics. The
consideration of the general metaplectic polynomials requires some additional
ideas, which we intend to address in the sequel to this paper.

\subsection{The Chinta-Gunnells action}

We recall now briefly the Chinta-Gunnells Weyl group action. See, e.g.,
\cite{BBCFH, BBF, CG07, CG, BF, CFG} and the nice overview article
\cite{BuOver} for the connection to Weyl group multiple Dirichlet series. Let
$W$ be the Weyl group of an irreducible root system $\Phi$, with Coxeter
generators $\left\{  s_{i}\right\} _{i=1}^{r} $ corresponding to a choice of
simple roots $\left\{  \alpha_{i}\right\} _{i=1}^{r} $. Let $P$ be the weight
lattice of $\Phi$. The Weyl group canonically acts on the fraction field
$\mathbb{C}(P)$ of the group algebra $\mathbb{C}[P]$ by field automorphisms.
Chinta and Gunnells have constructed a deformation of this action, which
depends on the choice of a $W$-invariant quadratic form $\mathbf{Q}%
:P\rightarrow\mathbb{Q}$ taking integer values on the root lattice $Q$ of
$\Phi$, a natural number $n$, and on parameters $v,g_{0},\ldots,g_{n-1}$
satisfying
\[
g_{0}=-1,\quad g_{j}g_{n-j}=v^{-1},\qquad j=1,\ldots,n-1.
\]

Let $0\leq\mathbf{r}_{m}\left(  j\right)  \leq m-1$ denote the remainder on
dividing $j$ by the natural number $m$, and define $g_{j}$ for arbitrary
$j\in\mathbb{Z}$ by setting $g_{j}=g_{r_{n}\left(  j\right)  }$, let
$\mathbf{B}\left(  \lambda,\mu\right)  =\mathbf{Q}\left(  \lambda+\mu\right)
-\mathbf{Q}\left(  \lambda\right)  -\mathbf{Q}\left(  \mu\right)  $ be the
bilinear form associated to $\mathbf{Q}$, and put $m\left(  \alpha\right)
=n/\gcd\left(  n,\mathbf{Q}\left(  \alpha\right)  \right)  $. It defines a new
root system $\Phi^{m}:=\{m(\alpha)\alpha\}_{\alpha\in\Phi}$, which is either
isomorphic to $\Phi$ or to $\Phi^{\vee}$. The weight lattice $P^{m}\subseteq
P$ of $\Phi^{m}$ is
\[
P^{m}=\{\lambda\in P \,\, | \,\, \mathbf{B}(\lambda,\alpha\bigr)\equiv0\,\,\,
\textup{mod } n \quad\forall\, \alpha\in\Phi\}
\]
(see Lemma \ref{relform2}). Then the Chinta-Gunnells action $\sigma_{i}%
=\sigma(s_{i})$ of the simple reflection $s_{i}\in W$ on $\mathbb{C}(P)$ is
given by the formula%
\begin{equation}
\label{sig_i}%
\begin{split}
& \sigma_{i}\left(  fx^{\lambda}\right)  :=\frac{(s_{i}f)x^{s_{i}\lambda}%
}{1-vx^{m\left(  \alpha_{i}\right)  \alpha_{i}}}\\
&  \!\!\times\left[  x^{-\mathbf{r}_{m\left(  \alpha_{i}\right)  }\left(
-\frac{\mathbf{B}\left(  \lambda,\alpha_{i}\right)  }{\mathbf{Q}\left(
\alpha_{i}\right)  }\right)  \alpha_{i}}\left(  1-v\right)  -vg_{\mathbf{Q}%
\left(  \alpha_{i}\right)  -\mathbf{B}\left(  \lambda,\alpha_{i}\right)
}x^{\left(  1-m\left(  \alpha_{i}\right)  \right)  \alpha_{i}}\left(
1-x^{m\left(  \alpha_{i}\right)  \alpha_{i}}\right)  \right]
\end{split}
\end{equation}
for $f\in\mathbb{C}(P^{m})$ and $\lambda\in P$. \noindent

It is non-trivial to show that the formula (\ref{sig_i}) defines a
representation of $W$. The main issue is to verify that the braid relations
are satisfied. Although this reduces to a rank $2$ computation, the
calculations become rather formidable, and in \cite{CG} the details are only
presented for $A_{2}$. Trying to find a natural interpretation of this
representation was one of the main motivations for our work.

Chinta and Gunnells \cite{CG} employed the action (\ref{sig_i}) to give an
explicit construction of the \textquotedblleft local\textquotedblright\ parts
of certain Weyl group multiple Dirichlet series, and to establish thus the
analytic continuation and functional equations for these series. In this
situation, the $g_{i}$ are $n$-th order Gauss sums for the local field, and
$v=p^{-1}$ with $p$ the cardinality of the residue field. Subsequently,
Chinta-Offen \cite{CO} for type $A$, and McNamara \cite{McN} in general,
showed that these local parts are essentially Whittaker functions for
principal series of certain $n$-fold \textquotedblleft
metaplectic\textquotedblright covers of quasi-split reductive groups. The
resulting explicit expression for the Whittaker function in terms of the
action \eqref{sig_i} is the metaplectic generalization of the
Casselman-Shalika formula. This result is in line with the fact that multiple
Dirichlet series should themselves be Whittaker coefficients attached to
metaplectic Eisenstein series \cite{BBF2, BF}.

Still more recently, Chinta-Gunnells-Puskas \cite{CGP} have shown that the $W
$-action (\ref{sig_i}) gives rise to a Cherednik \cite{Ch} type
Demazure-Lusztig action of the Hecke algebra of $W$. It leads to an expression
of the metaplectic Whittaker functions in terms of metaplectic
Demazure-Lusztig operators. Their work was partly motivated by
Brubaker-Bump-Licata \cite{BBL}, who gave formulas for (nonmetaplectic)
Iwahori-Whittaker functions in terms of Hecke operators and nonsymmetric
Macdonald polynomials. The recent work of Patnaik-Puskas \cite{PP} uses the
Chinta-Gunnells-Puskas Hecke algebra action to study metaplectic
Iwahori-Whittaker functions.

\subsection{Our results}

In Sections \ref{sec:reps} and \ref{sect4}, we give a uniform construction of a Weyl group representation
(Theorem \ref{mWrep}) and an associated Hecke algebra representation (Theorem
\ref{tauPropaction}) that generalize the Chinta-Gunnells \cite{CG} and
Chinta-Gunnells-Puskas \cite{CGP} representations, respectively. Our
construction does not involve case-by-case considerations, and it yields a
representation for the \emph{generic} Hecke algebra $H(\mathbf{k})$, which has
independent Hecke parameters for each root length in $\Phi$. Our method also
allows us to incorporate extra freedom in the definition of $g_{i}$ by
allowing them to depend on the root length (see Definition \ref{RepPar2} of
the \textit{representation parameters}). The Chinta-Gunnells and
Chinta-Gunnells-Puskas representations are recovered in the equal Hecke and
representation parameter case of our constructions.

Our starting point was the observation that (\ref{sig_i}) has many features in
common with formulas obtained by the process of \textquotedblleft
Baxterization\textquotedblright\ \cite{Ch}. The key idea behind this process
is that the group algebra of the \emph{affine }Weyl group and the
\textit{affine} Hecke algebra become \emph{isomorphic} after a suitable
localization, which allows one to relate certain representations of the two
algebras. This inspired our search for a natural representation of the affine
Hecke algebra whose associated localized affine Weyl group representation
produces \eqref{sig_i} for its $W$-action. Its first form can be recovered
from the Chinta-Gunnells-Puskas Hecke algebra action as follows.

Note that the Chinta-Gunnells $W$-action \eqref{sig_i} has an obvious
extension to a representation of the extended affine Weyl group $\widetilde{W}%
^{m}:=W\ltimes P^{m}$ with $\mu\in P^{m}$ acting on $\mathbb{C}(P)$ by
multiplication by $x^{\mu}$. Let $\widetilde{H}^{m}(\mathbf{k})$ be the
associated extended affine Hecke algebra with single Hecke parameter
$\mathbf{k}$ satisfying $\mathbf{k}^{2}=v$. If the affine extension of the
Chinta-Gunnells $W$-action on $\mathbb{C}(P)$ arises from a $\widetilde{H}%
^{m}(\mathbf{k})$-action on $\mathbb{C}[P]$ by localization, then the
generators $\{T_{i}\}_{i=1}^{r}$ of the finite Hecke algebra $H(\mathbf{k})$
act on $\mathbb{C}[P]$ by the Chinta-Gunnells-Puskas metaplectic
Demazure-Lusztig operators associated to $\sigma_{i}$ (cf. Proposition
\ref{tau_pol}). It follows that the underlying $H(\mathbf{k})$-representation
is equivalent to the $H(\mathbf{k})$-representation on $\mathbb{C}[P]$ defined
by
\begin{equation}
\label{pi_rep}\pi(T_{i})x^{\lambda}:=(\mathbf{k}-\mathbf{k}^{-1}%
)\overline{\nabla}_{i}(x^{\lambda}) -\mathbf{k}g_{-\mathbf{B}(\lambda
,\alpha_{i})}x^{s_{i}\lambda},\qquad\lambda\in P,
\end{equation}
with $\overline{\nabla}_{i}$ the following metaplectic version of the
divided-difference operator
\[
\overline{\nabla}_{i}(x^{\lambda}):=\frac{x^{\lambda}-x^{s_{i}\lambda
+\mathbf{r}_{m(\alpha_{i})} ((\lambda,\alpha_{i}^{\vee}))\alpha_{i}}%
}{1-x^{m(\alpha_{i})\alpha_{i}}}.
\]
But now we want to have an \textit{a priori} proof that \eqref{pi_rep} defines
a $H(\mathbf{k})$-action on $\mathbb{C}[P]$ and \textit{conclude} from it that
\eqref{sig_i} defines a $W$-action on $\mathbb{C}(P)$ via the localization technique.

Although the formulas \eqref{pi_rep} are much simpler than \eqref{sig_i}, a
direct case-by-case check that it defines a $H(\mathbf{k})$%
-re\-pre\-sen\-ta\-tion will be close to being as cumbersome as for the
Chinta-Gunnells action. Our first result is to circumvent the
case-by-case check by proving that $\pi$ is isomorphic to a quotient of the
induced module $\widetilde{H}^{m}(\mathbf{k})\otimes
_{H(\mathbf{k})}V_{C}$ for an appropriate $H(\mathbf{k})$-representation
$V_{C}$. This isomorphism in addition allows us to generalize 
$\pi$ and the Chinta-Gunnells Weyl group action to the context of 
generic affine Hecke algebras.

The $H(\mathbf{k})$-representation $V_{C}$ is defined as follows. Let
$V=\bigoplus_{\lambda\in P}\mathbb{C}v_{\lambda}$ be the complex vector space
with basis the weight lattice $P$. It has a natural left $H(\mathbf{k}%
)$-module structure reducing to the canonical $\mathbb{C}[W]$-module structure
when $\mathbf{k}=1$ (see Lemma \ref{refl_rep}). We call $V$ the reflection
representation of $H(\mathbf{k})$. For each $W$-invariant subset $D\subseteq
P$, the subspace $V_{D}:=\bigoplus_{\lambda\in D}\mathbb{C}v_{\lambda}$ is a
$H(\mathbf{k})$-submodule of $V$. In particular, $V_{\{0\}}$ is the trivial
representation of $H(\mathbf{k})$. The appropriate choice of $W$-invariant
subset $C$ of $P$ in the above realization of $\pi$ now turns out to be
\[
C:=\{\lambda\in P \,\, | \,\, (\lambda,\alpha^{\vee})\leq m(\alpha)
\quad\forall\,\, \alpha\in\Phi\}.
\]
Note that $C$ contains a complete set of coset representatives of $P/P^{m}$.

The following trivial example is instructive to get a feeling for what is
going on. Suppose that $m(\alpha)=1$ for all $\alpha\in\Phi$. Then $P^{m}=P$
and $\overline{\nabla}_{i}$ is the standard divided-difference operator on
$\mathbb{C}[P]$. In this case it is well known that \eqref{pi_rep} is
equivalent to the induced module $\widetilde{H}^{m}(\mathbf{k})\otimes
_{H(\mathbf{k})}V_{\{0\}}$ by the Bernstein-Zelevinsky \cite{Lu} presentation
of $\widetilde{H}^{m}(\mathbf{k})$. The $W$-subset $C$ in this case is
oversized, with $C\setminus\{0\}$ being the set of nonzero miniscule weights
in $P$.

In Section \ref{sect5} we construct the metaplectic polynomials in type $A$. The extension to arbitrary types will be treated in a forthcoming paper.  The $\textup{GL}_r$ double affine Hecke algebra
$\mathbb{H}^{(m)}$ has generators $T_0,\ldots,T_{r-1},\omega^{\pm 1}, x_1^{\pm m},\ldots, x_r^{\pm m}$, with $T_0,\ldots,T_{r-1},\omega^{\pm 1}$ Coxeter type generators of a copy of the $\textup{GL}_r$ affine Hecke algebra in $\mathbb{H}^{(m)}$ ($\omega$ is the generator of abelian group of group elements of length zero), and $T_1,\ldots,T_{r-1},x_1^{\pm m},\ldots,x_r^{\pm m}$ Bernstein-Zelevinsky type generators of the second copy of the $\textup{GL}_r$ affine Hecke algebra in
$\mathbb{H}^{(m)}$ (the $x_j^{\pm m}$ ($j=1,\ldots,r$) are generating its commutative subalgebra). The metaplectic representation of the second copy of the $\textup{GL}_r$
affine Hecke algebra is acting on Laurent polynomials in  $x_1^{\pm 1},\dots,x_r^{\pm 1}$, where $T_i$ for $1 \leq i < r$ act by \eqref{pi_rep} and $x^{\nu}$ ($\nu \in m\mathbb{Z}^{r}$) act by multiplication. It extends to a representation $\widehat{\pi}$ of $\mathbb{H}^m$, with $\omega$ acting as a twisted-cyclic permutation of the variables and 
$T_0$ by an appropriate affine version of the metaplectic Demazure-Lusztig operator
(see Theorem \ref{GLrmetabasic}). The representation $\widehat{\pi}$ is a metaplectic generalization of Cherednik's basic representation \cite{Ch,Ma}, which we call the metaplectic basic representation.  

The $\textup{GL}_r$ affine Hecke algebra generated by $T_0,\ldots,T_{r-1},\omega^{\pm 1}$ in its Bernstein-Zelevinsky presentation provides an abelian subalgebra generated by
elements $Y^{\pm m\epsilon_i}$ ($i=1,\ldots,r$). We define the metaplectic polynomials $E_\mu^{(m)}$ ($\mu\in\mathbb{Z}^r$) in Theorem \ref{thm:E-pols} as the simultaneous eigenfunctions of $\widehat{\pi}(Y^{m\epsilon_i})$ ($i=1,\ldots,r$).
It depends, besides on the standard Macdonald parameters, on the additional representation parameters $g_j$. The subfamily indexed by $m\mathbb{Z}^r$ recovers the nonsymmetric Macdonald polynomials in the variables $x_1^m,\ldots,x_r^m$ (see Remark \ref{classical_case}).  At the end of Section \ref{sect5} we provide examples of $GL_{3}$-metaplectic polynomials, highlighting some important phenomena.  In a followup paper other important properties, such as triangularity and orthogonality will be established in the context of arbitrary root systems.

\subsection{The structure of the paper}

We now briefly discuss the content of the paper. We introduce in Section
\ref{sect2} the appropriate metaplectic structures on the root systems and
affine Weyl and Hecke algebras. Section \ref{sec:reps} is devoted to the
metaplectic representation theory of the affine Weyl groups and generic affine
Hecke algebras. We introduce the reflection representation in Subsection
\ref{subs31}. Subsection \ref{subs32} forms the heart of our approach: we
introduce the analogue of \eqref{pi_rep} for generic Hecke and representation
parameters and establish that it defines a representation of the generic
affine Hecke algebra by identifying it with a quotient of the induced module
$\widetilde{H}^{m}(\mathbf{k})\otimes_{H(\mathbf{k})}V_{C}$ (see Theorem
\ref{mainTHM}). In Subsection \ref{subs33} we explain the localization
technique and apply it to $\pi$ (Theorem \ref{mainTHM}) to obtain the
generalized Chinta-Gunnells $W$-action (Theorem \ref{mWrep}).

In Section \ref{sect4} we form the associated metaplectic Demazure-Lusztig
operators and generalize some of the results from \cite{CGP} to the setting of
unequal Hecke and representation parameters. We also simplify some of the
proofs from that paper by using the standard symmetrizer and antisymmetrizer
elements in the Hecke algebra. This allows us to define a natural class of
\textquotedblleft Whittaker functions\textquotedblright\ for generic Hecke
algebras. It is natural to ask whether these more general functions arise as
actual matrix coefficients for some class of representations of $p$-adic
groups. This question is of particular interest since generic Hecke algebras
have begun to play an increasing role in the study of the Bernstein components
within the categories of smooth representations of $p$-adic groups, see, e.g.,
\cite{BK, H} and references therein.

In Section \ref{sect5}, we construct the metaplectic polynomials in type $A$.  We begin by setting up notations and modifications specific to the $GL_r$ case.  The double affine Hecke algebra $\mathbb{H}^{m}$ is presented in Section \ref{sect:DAHA}, and the metaplectic basic representation is in Section \ref{sect:metbasic}.   The characterization of the metaplectic polynomials as eigenfunctions of the metaplectic operators $\widehat{\pi}(Y^\nu)$ ($\nu\in m\mathbb{Z}^r$) may be found in Section \ref{sect:metpols}.  We also discuss the dependence on parameters, showing that we do not lose any generality by taking the quadratic form $\mathbf{Q}$ to satisfy $\mathbf{Q}(\alpha) = 1$ for $\alpha$ a root.  Finally, in the Appendix, we provide a list of metaplectic polynomials for $r=3$ and $1 \leq m \leq 5$.

%We briefly mention another related avenue of research that we plan to
%investigate further. The Whittaker functions mentioned above arise from the
%representation of $\widetilde{H}^{m}(\mathbf{k})$ via the antisymmetrizer (cf.
%the proof of Theorem \ref{WhittakerDL}). One may also consider the
%corresponding symmetric variant. These polynomials are indexed by dominant
%weights $\lambda\in P^{+}$, symmetric with respect to the Chinta-Gunnells $W
%$-action, and for $\lambda\in P^{m+}$ reduce to Hall-Littlewood polynomials
%for the root system $\Phi^{m}$ (see Remark \ref{metHL}). In a work in
%progress, we investigate these further and construct analogues of Macdonald
%polynomials in this setting, along the lines of the standard
%Cherednik-Macdonald theory.

Let us conclude with remarking that the localization procedure we use in this
paper is instrumental in Cherednik's construction of quantum affine
Knizhnik-Zamolodchikov equations attached to affine Hecke algebra modules.
Closely related to it is the role of the localization procedure for type $A$
in the context of integrable vertex models with $U_{q}(\widehat{\mathfrak{sl}%
}_{n})$-symmetry, in the special cases that the associated braid group action
descends to an affine Hecke algebra action, in which case the localization
procedure is often referred to as Baxterization (see, e.g., \cite{Ch,S} and
references therein). This is exactly the context in which the metaplectic
Whittaker function can be realized as a partition function, the corresponding
integrable model being \textquotedblleft metaplectic ice\textquotedblright,
see \cite{BBB, BBBF, BBBG}. It is an intriguing open question whether there is
a conceptual connection with the current interpretation of the Chinta-Gunnells
action through localization.\newline

\noindent\textbf{Acknowledgement:} The research of the first author was
partially supported by a Simons Foundation grant (509766). Part of this work
was initiated during the Workshop on Hecke Algebras and Lie Theory at the
University of Ottawa, co-organized by the first author and attended by the
second; both authors thank the National Science Foundation (DMS-162350), the
Fields Institute, and the University of Ottawa for funding this workshop.

%%%%%%%%%%%%%%%%%%%%%%%%%%%%%%%%%%%%%%%%%%

\section{The extended affine Hecke algebra}

\label{sect2}
%%%%%%%%%%%%%%%%%%%%%%%%%%%%%%%%%%%%%%%%%%

\subsection{The root system}

Let $E$ be an Euclidean space with scalar product $(\cdot,\cdot)$ and norm
$\|\cdot\|$. Let $\Phi\subset E$ be an irreducible reduced root system, and
$W\subset O(E)$ its Weyl group. The reflection in $\alpha\in\Phi$ is denoted
by $s_{\alpha}\in W$, and its co-root is $\alpha^{\vee}:=2\alpha
/\|\alpha\|^{2}$.

Fix a base $\{\alpha_{1},\ldots,\alpha_{r}\}$ of $\Phi$. Let $\Phi^{+}$ be the
corresponding set of positive roots and write $s_{i}:=s_{\alpha_{i}}$ for
$i=1,\ldots,r$. Let
\[
P:=\{\lambda\in V \,\, | \,\, (\lambda,\alpha^{\vee})\in\mathbb{Z}\quad
\forall\alpha\in\Phi\}= \bigoplus_{i=1}^{r}\mathbb{Z}\varpi_{i}%
\]
be the weight lattice of $\Phi$ with $\varpi_{i}\in E$ the fundamental
weights, defined by $(\varpi_{i},\alpha_{j}^{\vee})=\delta_{i,j}$. Let
\[
Q=\mathbb{Z}\Phi=\bigoplus_{i=1}^{r}\mathbb{Z}\alpha_{i}%
\]
be the root lattice of $\Phi$.

%%%%%%%%%%%%%%%%%%%%%%%%%%%%%%%%%%%%

\subsection{The metaplectic structure}

%%%%%%%%%%%%%%%%%%%%%%%%%%%%%%%%%%%%
In the theory of metaplectic Whittaker functions, a new root system $\Phi^{m}$
is attached to the metaplectic covering data of the reductive group over the
non-archimedean local field, cf. \cite{CG,CGP} and references therein. We
recall in this subsection this additional metaplectic data on the root system.

Fix a $W$-invariant quadratic form $\mathbf{Q}: P \rightarrow\mathbb{Q}$ which
takes integral values on $Q$ and write $\mathbf{B}: P\times P\rightarrow
\mathbb{Q}$ for the associated symmetric bilinear pairing
\[
\mathbf{B}(\lambda,\mu):=\mathbf{Q}(\lambda+\mu)-\mathbf{Q}(\lambda
)-\mathbf{Q}(\mu),\qquad\lambda,\mu\in P.
\]
Then $\mathbf{Q}(\cdot)=\frac{\kappa}{2}\|\cdot\|^{2}$ for some $\kappa
\in\mathbb{R}^{\times}$, and hence $\mathbf{B}(\lambda,\mu)=\kappa(\lambda
,\mu)$ for all $\lambda,\mu\in P$. In particular, for all $\lambda\in P$ and
$\alpha\in\Phi$,
\begin{equation}
\label{relin}\frac{\mathbf{B}(\lambda,\alpha)}{\mathbf{Q}(\alpha)}%
=(\lambda,\alpha^{\vee}).
\end{equation}

Let $n\in\mathbb{Z}_{>0}$ and define
\[
m(\alpha):=\frac{n}{\textup{gcd}(n,\mathbf{Q}(\alpha))}=\frac{\textup{lcm}%
(n,\mathbf{Q}(\alpha))}{\mathbf{Q}(\alpha)} \qquad\forall\, \alpha\in\Phi.
\]
Note that $m: \Phi\rightarrow\mathbb{Z}_{>0}$ is $W$-invariant.

Set $\Phi^{m}:=\{\alpha^{m}:=m(\alpha)\alpha\}_{\alpha\in\Phi}\subset E$. Then
$\Phi^{m}$ is a root system. In fact, if $m$ is constant then $\Phi^{m}$ is
isomorphic to $\Phi$, while if $m$ is nonconstant then $\Phi^{m}$ is
isomorphic to the co-root system $\Phi^{\vee}=\{\alpha^{\vee}\}_{\alpha\in
\Phi}$ (this follows from the definition of $m(\alpha)$ and the fact that
$\mathbf{Q}(\cdot) =\frac{\kappa}{2}\|\cdot\|^{2}$). In particular,
$\{\alpha_{1}^{m},\ldots,\alpha_{r}^{m}\}$ is a base of $\Phi^{m}$ and $W$ is
the Weyl group of $\Phi^{m}$.

Write $Q^{m}$ for the root lattice of $\Phi^{m}$ and $P^{m}$ for the weight
lattice of $\Phi^{m}$. Since $(\alpha^{m})^{\vee}=m(\alpha)^{-1}\alpha^{\vee}$
for $\alpha\in\Phi$, we have
\[
Q^{m}=\bigoplus_{i=1}^{r}\mathbb{Z}\alpha_{i}^{m},\qquad P^{m}=\bigoplus
_{i=1}^{r}\mathbb{Z}\varpi_{i}^{m}%
\]
with $\varpi_{i}^{m}:=m(\alpha_{i})\varpi_{i}$ the fundamental weights of
$P^{m}$.
%%%%%%%%%%%%%%%%%%%%%%%%%%%%%%%%%%%%%%%%%%%%

\begin{lem}
\label{relform} \textbf{a.} For $\alpha\in\Phi$ and $\lambda\in P$ we have
\[
\mathbf{B}(\lambda,\alpha^{m})=\textup{lcm}(n,\mathbf{Q}(\alpha))\bigl(\lambda
,\alpha^{\vee}).
\]
\textbf{b.} For $\alpha\in\Phi$ and $\lambda\in P$ we have
\[
\mathbf{B}(\lambda,\alpha)\equiv0\,\,\, \textup{\ mod } n\,\, \Leftrightarrow
(\lambda,\alpha^{\vee})\equiv0\,\,\, \textup{\ mod } m(\alpha).
\]

\end{lem}

%%%%%%%%%%%%%%%%%%%%%%%%%%%%%%%%%%%%%%%%%%%%

\begin{proof}
\textbf{a} For $\lambda\in P$ and $\alpha\in\Phi$ we have
\[%
\begin{split}
\mathbf{B}(\lambda,\alpha^{m}) & =m(\alpha)\mathbf{B}(\lambda,\alpha)\\
& =\frac{n\, \mathbf{B}(\lambda,\alpha)}{\textup{gcd}(n,\mathbf{Q}(\alpha))}\\
& =\frac{n\mathbf{Q}(\alpha)}{\textup{gcd}(n,\mathbf{Q}(\alpha))}%
\frac{\mathbf{B}(\lambda,\alpha)}{\mathbf{Q}(\alpha)} =\textup{lcm}%
(n,\mathbf{Q}(\alpha))(\lambda,\alpha^{\vee}).
\end{split}
\]
\textbf{b.} For $\lambda\in P$ and $\alpha\in\Phi$ we have
\[%
\begin{split}
\mathbf{B}(\lambda,\alpha)\equiv0\,\,\, \textup{\ mod } n\,\,  &
\Leftrightarrow\,\, \mathbf{Q}(\alpha)(\lambda,\alpha^{\vee})\equiv0\,\,\,
\textup{\ mod } n\\
& \Leftrightarrow\,\, (\lambda,\alpha^{\vee})\equiv0\,\,\, \textup{\ mod }
\frac{n}{\textup{gcd}(n,\mathbf{Q}(\alpha))}\\
& \Leftrightarrow\,\, (\lambda,\alpha^{\vee})\equiv0\,\,\, \textup{\ mod }
m(\alpha). \qedhere
\end{split}
\]

\end{proof}

%%%%%%%%%%%%%%%%%%%%%%%%%%%%%%%%%%%%%%%%%%%%%

\begin{lem}
\label{relform2}
\[%
\begin{split}
P^{m} & =\{\lambda\in P\,\, | \,\, (\lambda,\alpha^{\vee})\equiv0\,\,\,
\textup{mod}\,\, m(\alpha)\,\,\, \forall\, \alpha\in\Phi\}\\
& =\{\lambda\in P \,\, | \,\, \mathbf{B}(\lambda,\alpha)\equiv0\,\,\,
\textup{mod }\, n\,\,\, \forall\, \alpha\in\Phi\}.
\end{split}
\]

\end{lem}

%%%%%%%%%%%%%%%%%%%%%%%%%%%%%%%%%%%%%%%%%%%%%

\begin{proof}
The first equality follows from the fact that $(\alpha^{m})^{\vee}%
=m(\alpha)^{-1}\alpha^{\vee}$ for $\alpha\in\Phi$. The second equality follows
immediately from part \textbf{b} of Lemma \ref{relform}.
\end{proof}

%%%%%%%%%%%%%%%%%%%%%%%%%%%%%%%%%%%%%%%%%

%%%%%%%%%%%%%%%%%%%%%%%%%%%%%%%%%%%

\subsection{The extended affine Hecke algebra}

\label{ssec:eaha}
%%%%%%%%%%%%%%%%%%%%%%%%%%%%%%%%%

We start with the definition of the finite Hecke algebra. Let $\mathbf{k}:
\Phi\rightarrow\mathbb{C}^{\times}$ be a $W$-invariant function and write
$\mathbf{k}_{\alpha}$ for the value of $\mathbf{k}$ at $\alpha\in\Phi$. Set
$\mathbf{k}_{i}:= \mathbf{k}_{\alpha_{i}}$ for $i=1,\ldots,r$.

%%%%%%%%%%%%%%%%%%%%%%%%%%%

\begin{defi}
\label{H} The Hecke algebra $H(\mathbf{k})$ associated to the root system
$\Phi$ is the unital associative algebra over $\mathbb{C}$ generated by
$T_{1},\ldots,T_{r}$ with defining relations

\begin{enumerate}
\item[\textbf{a.}] $(T_{i}-\mathbf{k}_{i})(T_{i}+\mathbf{k}_{i}^{-1})=0$ for
$i=1,\ldots,r$,

\item[\textbf{b.}] For $1\leq i\not =j\leq r$ the braid relation $T_{i}%
T_{j}T_{i}\cdots=T_{j}T_{i}T_{j}\cdots$ ($m_{ij}$ factors on each side, with
$m_{ij} $ the order of $s_{i}s_{j}$ in $W$).
\end{enumerate}
\end{defi}

%%%%%%%%%%%%%%%%%%%%%%%%%%%
Define the length of $w\in W$ by
\[
\ell(w):=\#\bigl(\Phi^{+}\cap w^{-1}\Phi^{-}\bigr).
\]
For $w=s_{i_{1}}\cdots s_{i_{\ell}}$ ($1\leq i_{j}\leq r$) a reduced
expression of $w\in W$ (i.e. $\ell=\ell(w)$), set
\[
T_{w}:=T_{i_{1}}\cdots T_{i_{t}}\in H(\mathbf{k}).
\]
The $T_{w}$ ($w\in W$) are well defined and form a linear basis of
$H(\mathbf{k})$.

We now introduce the extended affine Hecke algebra $\widetilde{H}%
^{m}(\mathbf{k})$ associated to the finite root system $\Phi^{m}$ through its
Bernstein-Zelevinsky presentation (see \cite{Lu}). It contains as subalgebras
the finite Hecke algebra $H(\mathbf{k})$ and the group algebra $\mathbb{C}%
[P^{m}]$ of the weight lattice $P^{m}$ of $\Phi^{m}$. We write the canonical
basis elements of $\mathbb{C}[P^{m}]$ in multiplicative form $x^{\mu}$
($\mu\in P^{m}$), so that $x^{\mu}x^{\nu}=x^{\mu+\nu}$ and $x^{0}=1$. The Weyl
group $W$ acts naturally on $\mathbb{C}[P^{m}]$ by algebra automorphisms.

For $1\leq i\leq r$ there exists a well defined linear operator $\nabla
_{i}^{m}$ on $\mathbb{C}[P^{m}]$ satisfying
\[
\nabla_{i}^{m}(x^{\nu}):=\frac{x^{\nu}-x^{s_{i}\nu}}{1-x^{\alpha_{i}^{m}}}%
\]
for $\nu\in P^{m}$ (note that $x^{\nu}-x^{s_{i}\nu}$ is divisible by
$1-x^{\alpha_{i}^{m}}$ in $\mathbb{C}[P^{m}]$). It is called the divided
difference operator associated to the simple root $\alpha_{i}^{m}$.

%%%%%%%%%%%%%%%%%%%%%%%%%%%

\begin{defi}
\label{tildeH} The extended affine Hecke algebra $\widetilde{H}^{m}%
(\mathbf{k}) $ is the unital associative algebra over $\mathbb{C}$ generated
by the algebras $H(\mathbf{k})$ and $\mathbb{C}[P^{m}]$, with additional
defining relations
\begin{equation}
\label{BL1}T_{i}x^{\nu}-x^{s_{i}\nu}T_{i}=(\mathbf{k}_{i}-\mathbf{k}_{i}%
^{-1})\nabla_{i}^{m}(x^{\nu})
\end{equation}
for $i=1,\ldots,r$ and $\nu\in P^{m}$.
\end{defi}

%%%%%%%%%%%%%%%%%%%%%%%%%%%

It is well known that the multiplication map defines a linear isomorphism
\begin{equation}
\label{BZiso}H(\mathbf{k})\otimes\mathbb{C}[P^{m}]\overset{\sim
}{\longrightarrow} \widetilde{H}^{m}(\mathbf{k}).
\end{equation}

%%%%%%%%%%%%%%%%%%%%%%%%%%%%%%%%%%%%%%

\section{Metaplectic representations}

\label{sec:reps}
%%%%%%%%%%%%%%%%%%%%%%%%%%%%%%%%%%%%%%%

\subsection{The reflection representation of $H(\mathbf{k})$}

\label{subs31}
%%%%%%%%%%%%%%%%%%%%%%%%%%%%%%%%%%%%%%
Set
\[
V:=\bigoplus_{\lambda\in P}\mathbb{C}v_{\lambda}.
\]
It inherits a left $W$-action by the linear extension of the canonical action
of $W$ on $P$. For a $W$-invariant subset $D\subset P$ we write
\[
V_{D}:=\bigoplus_{\lambda\in D}\mathbb{C}v_{\lambda}%
\]
for the corresponding $W$-submodule of $V$. Then $V=\bigoplus_{\lambda\in
P^{+}}V_{\mathcal{O}_{\lambda}}$ with $\mathcal{O}_{\lambda}=W\lambda$ the
$W$-orbit of $\lambda$ in $P$ and $P^{+}\subset P$ the cone of dominant
weights of $\Phi$ with respect to the base $\{\alpha_{1},\ldots,\alpha_{r}\}$.
In this subsection we deform the $W$-action on $V_{D}$ and $V$ to a
$H(\mathbf{k})$-action.

Fix $\lambda\in P^{+}$. The stabilizer subgroup
\[
W_{\lambda}:=\{w\in W \,\, | \,\, w\lambda=\lambda\}
\]
is a standard parabolic subgroup of $W$. It is generated by the simple
reflections $s_{i}$ ($i\in I_{\lambda}$), with $I_{\lambda}$ the index subset
\[
I_{\lambda}:=\{i\in\{1,\ldots,r\}\,\, | \,\, s_{i}\lambda=\lambda\}.
\]
Note that $V_{\mathcal{O}_{\lambda}}\simeq\mathbb{C}[W]\otimes_{\mathbb{C}%
[W_{\lambda}]}\mathbb{C}$ as $W$-modules, with $\mathbb{C}$ regarded as the
trivial $W_{\lambda}$-module. This description leads to a natural Hecke
deformation of the $W$-action on $V_{\mathcal{O}_{\lambda}}$ as follows.

Let $W^{\lambda}$ be the minimal coset representatives of $W/W_{\lambda}$,
which can be characterized by
\[
W^{\lambda}=\{w\in W \,\, | \,\, \ell(ws_{i})=\ell(w)+1\quad\forall\, i\in
I_{\lambda}\}.
\]
For $\lambda\in P^{+}$ let $H_{\lambda}(\mathbf{k})\subset H(\mathbf{k})$ be
the subalgebra generated by the $T_{i}$ ($i\in I_{\lambda}$). Write
$\mathbb{C}_{\lambda}$ for the trivial one-dimensional $H_{\lambda}%
(\mathbf{k})$-module defined by $T_{i}\mapsto\mathbf{k}_{i}$ ($i\in
I_{\lambda}$). Consider the linear isomorphism
\[
\phi_{\lambda}: H(\mathbf{k})\otimes_{H_{\lambda}(\mathbf{k})}\mathbb{C}%
_{\lambda}\overset{\sim}{\longrightarrow} V_{\mathcal{O}_{\lambda}}%
\]
defined by $\phi_{\lambda}(T_{w}\otimes_{H_{\lambda}(\mathbf{k})}%
1)=v_{w\lambda}$ for $w\in W^{\lambda}$. Transporting the canonical
$H(\mathbf{k})$-module structure on $H(\mathbf{k})\otimes_{H_{\lambda
}(\mathbf{k})}\mathbb{C}_{\lambda}$ to $V_{\mathcal{O}_{\lambda}}$ through the
linear isomorphism $\phi_{\lambda}$ turns $V_{\mathcal{O}_{\lambda}}$ into a
$H(\mathbf{k})$-module. The resulting direct sum $H(\mathbf{k})$-module
structure on $V=\bigoplus_{\lambda\in P^{+}}V_{\mathcal{O}_{\lambda}}$ can be
explicitly described as follows.
%%%%%%%%%%%%%%%%%%%%%%%%%%%%%%%%%%%%%%%%

\begin{lem}
\label{refl_rep} For $\mu\in P$ we have
\[
T_{i}v_{\mu}=
\begin{cases}
v_{s_{i}\mu}\quad & \hbox{ if }\quad(\mu,\alpha_{i}^{\vee})>0,\\
\mathbf{k}_{i}v_{\mu}\quad & \hbox{ if } \quad(\mu,\alpha_{i}^{\vee})=0,\\
(\mathbf{k}_{i}-\mathbf{k}_{i}^{-1})v_{\mu}+v_{s_{i}\mu}\quad & \hbox{ if }
\quad(\mu,\alpha_{i}^{\vee})<0
\end{cases}
\]
for $i=1,\ldots,r$.
\end{lem}

%%%%%%%%%%%%%%%%%%%%%%%%%%%%%%%%%%%%%%%%

\begin{proof}
Write $\mu=w\lambda$ with $\lambda\in P^{+}$ and $w\in W^{\lambda}$. We claim that

\begin{enumerate}
\item[\textbf{(1)}] $\bigl(\mu,\alpha_{i}^{\vee}\bigr)>0$ $\Leftrightarrow$
$\ell(s_{i}w)=\ell(w)+1$ and $s_{i}w\in W^{\lambda}$;

\item[\textbf{(2)}] $\bigl(\mu,\alpha_{i}^{\vee}\bigr)=0$ $\Leftrightarrow$
$\ell(s_{i}w)=\ell(w)+1$ and $s_{i}w\not \in W^{\lambda}$;

\item[\textbf{(3)}] $\bigl(\mu,\alpha_{i}^{\vee}\bigr)<0$ $\Leftrightarrow$
$\ell(s_{i}w)=\ell(w)-1$. In this case we have $s_{i}w\in W^{\lambda}$.
\end{enumerate}

Since each $w\in W^{\lambda}$ satisfies exactly one of these three conditions,
it suffices to prove the $\Leftarrow$'s.\newline\textbf{Case (1):} $\ell
(s_{i}w)=\ell(w)+1$ and $s_{i}w\in W^{\lambda}$. Then $\ell(s_{i}w)=\ell(w)+1$
implies $w^{-1}\alpha_{i}\in\Phi^{+}$, hence $(\mu,\alpha_{i}^{\vee}%
)=(\lambda,w^{-1}\alpha_{i}^{\vee})\geq0$. The assumption $s_{i}w\in
W^{\lambda}$ implies $s_{i}w\lambda\not =w\lambda$, in particular $(\mu
,\alpha_{i}^{\vee})\not =0$. Hence $(\mu,\alpha_{i}^{\vee})>0$.\newline%
\textbf{Case (2):} $\ell(s_{i}w)=\ell(w)+1$ and $s_{i}w\not \in W^{\lambda}$.
Then $s_{i}w=ws_{j}$ for some $j\in I_{\lambda}$ by \cite[Lem. 3.2]{D}. Hence
$s_{i}w\lambda=w\lambda$ and consequently $(\mu,\alpha_{i}^{\vee})=0$%
.\newline\textbf{Case (3):} $\ell(s_{i}w)=\ell(w)-1$. Then $(\mu,\alpha
_{i}^{\vee})=(\lambda,w^{-1}\alpha_{i}^{\vee})\leq0$ since $w^{-1}\alpha
_{i}\in\Phi^{-}$. If $s_{i}w\lambda=w\lambda$ then $s_{i}w$ would be a
representative of $wW_{\lambda}$ of small length than $w$, which is absurd.
Hence $s_{i}w\lambda\not =w\lambda$, and consequently $(\mu,\alpha_{i}^{\vee
})<0$. If $s_{i}w\not \in W^{\lambda}$ then the minimal length representative
$w^{\prime}\in W^{\lambda}$ of the coset $s_{i}wW_{\lambda}$ has length
strictly smaller than $\ell(s_{i}w)=\ell(w)-1$. But then $wW_{\lambda}$
contains an element of length strictly smaller than $\ell(w)$, which is
absurd. Hence $s_{i}w\in W^{\lambda}$.

It is now easy to conclude the proof of the lemma:\newline\textbf{Case (1):}
$\ell(s_{i}w)=\ell(w)+1$ and $s_{i}w\in W^{\lambda}$. Then
\[
T_{i}v_{\mu}=\phi_{\lambda}\bigl(T_{i}T_{w}\otimes_{H_{\lambda}(\mathbf{k}%
)}1\bigr)=\phi_{\lambda}\bigl(T_{s_{i}w}\otimes_{H_{\lambda}(\mathbf{k}%
)}1\bigr)=v_{s_{i}\mu}.
\]
\textbf{Case (2):} $\ell(s_{i}w)=\ell(w)+1$ and $s_{i}w\not \in W^{\lambda}$.
Let $j\in I_{\lambda}$ such that $s_{i}w=ws_{j}$. Note that $\alpha_{j}\in
W\alpha_{i}$, hence $\mathbf{k}_{i}=\mathbf{k}_{j}$, and that $\ell
(ws_{j})=\ell(w)+1$, so that $T_{i}T_{w}=T_{s_{i}w}=T_{ws_{j}}=T_{w}T_{j}$.
Then
\[
T_{i}v_{\mu}=\phi_{\lambda}\bigl(T_{i}T_{w}\otimes_{H_{\lambda}(\mathbf{k}%
)}1\bigr)= \phi_{\lambda}\bigl(T_{w}T_{j}\otimes_{H_{\lambda}(\mathbf{k}%
)}1\bigr)=\mathbf{k}_{i}v_{\mu}.
\]
\textbf{Case (3):} $l(s_{i}w)=l(w)-1$. Using $T_{i}^{2}=(\mathbf{k}%
_{i}-\mathbf{k}_{i}^{-1})T_{i}+1$ we get
\[
T_{i}T_{w}=T_{i}^{2}T_{s_{i}w}=(\mathbf{k}_{i}-\mathbf{k}_{i}^{-1}%
)T_{w}+T_{s_{i}w}%
\]
in $H(\mathbf{k})$, and hence
\[
T_{i}v_{\mu}=(\mathbf{k}_{i}-\mathbf{k}_{i}^{-1})v_{\mu}+v_{s_{i}\mu}%
\]
since $s_{i}w\in W^{\lambda}$.
\end{proof}

%%%%%%%%%%%%%%%%%%%%%%%%%%%%%%%%%%%%%%%%%%%%%%%%

\subsection{The metaplectic affine Hecke algebra representation}

\label{subs32}
%%%%%%%%%%%%%%%%%%%%%%%%%%%%%%

For $s\in\mathbb{Z}_{>0}$ and $t\in\mathbb{Z}$ let $\mathbf{r}_{s}%
(t)\in\{0,\ldots,s-1\}$ be the remainder of $t$ modulo $s$. Define
$\mathbf{q},\mathbf{r}: P\rightarrow P$ by
\[%
\begin{split}
\mathbf{r}(\lambda) & :=\sum_{i=1}^{r}\mathbf{r}_{m(\alpha_{i})}%
\bigl((\lambda,\alpha_{i}^{\vee})\bigr)\varpi_{i},\\
\mathbf{q}(\lambda) & :=\lambda-\mathbf{r}(\lambda).
\end{split}
\]
%%%%%%%%%%%%%%%%%%%%%%%%%%%%%

\begin{lem}
\label{Proj} $\mathbf{q}(P)\subseteq P^{m}$.
\end{lem}

%%%%%%%%%%%%%%%%%%%%%%%%%%%%%

\begin{proof}
For $i=1,\ldots,r$ and $\lambda\in P$ we have
\[%
\begin{split}
\bigl(\mathbf{q}(\lambda),\alpha_{i}^{m\vee}\bigr)  & =m(\alpha_{i}%
)^{-1}(\mathbf{q}(\lambda),\alpha_{i}^{\vee})\\
& =m(\alpha_{i})^{-1}\bigl((\lambda,\alpha_{i}^{\vee})- \mathbf{r}%
_{m(\alpha_{i})}((\lambda,\alpha_{i}^{\vee}))\bigr)\in\mathbb{Z}. \qedhere
\end{split}
\]

\end{proof}

%%%%%%%%%%%%%%%%%%%%%%%%%%%%%%

Let $\mathbb{C}[P]=\textup{span}\{x^{\lambda}\}_{\lambda\in P}$ be the group
algebra of the weight lattice $P$. The Weyl group $W$ acts naturally on
$\mathbb{C}[P]$ by algebra automorphisms.

Note that the divided difference operator $\widetilde{\nabla}_{i}^{m}$
featuring in the Bernstein-Zelevinsky cross relations \eqref{BL1} of the
extended affine Hecke algebra $\widetilde{H}(\mathbf{k})$ satisfies
\[
\nabla_{i}^{m}(x^{\nu}):=\frac{x^{\nu}-x^{s_{i}\nu}}{1-x^{\alpha_{i}^{m}}}=
\left( \frac{1-x^{-(\nu,\alpha_{i}^{m\vee})\alpha_{i}^{m}}}{1-x^{\alpha
_{i}^{m}}}\right) x^{\nu},\qquad\nu\in P^{m}%
\]
for $i=1,\ldots,r$.

%%%%%%%%%%%%%%%%%%%%%%%%%%%%%%%%%%%%%%%

\begin{lem}
For $i=1,\ldots,r$ there exists a unique linear map
\[
\overline{\nabla}_{i}: \mathbb{C}[P]\rightarrow\mathbb{C}[P]
\]
satisfying
\begin{equation}
\label{nablabar}\overline{\nabla}_{i}(x^{\lambda}):=\left( \frac
{1-x^{-(\mathbf{q}(\lambda),\alpha_{i}^{m\vee})\alpha_{i}^{m}}} {1-x^{\alpha
_{i}^{m}}}\right) x^{\lambda}%
\end{equation}
for $\lambda\in P$. Furthermore,
\[
\overline{\nabla}_{i}|_{\mathbb{C}[P^{m}]}=\nabla_{i}^{m}.
\]

\end{lem}

%%%%%%%%%%%%%%%%%%%%%%%%%%%%%%%%%%%

\begin{proof}
Note that $\overline{\nabla}_{i}: \mathbb{C}[P]\rightarrow\mathbb{C}[P]$ is a
well defined linear operator by the previous lemma. In fact,
\begin{equation}
\label{111}\overline{\nabla}_{i}(x^{\lambda}):=
\begin{cases}
-x^{\lambda-\alpha_{i}^{m}}-\cdots-x^{\lambda-(\mathbf{q}(\lambda),\alpha
_{i}^{m\vee})\alpha_{i}^{m}}, \quad & \hbox{ if } (\mathbf{q}(\lambda
),\alpha_{i}^{m\vee})>0,\\
0\quad & \hbox{ if } (\mathbf{q}(\lambda),\alpha_{i}^{m\vee})=0,\\
x^{\lambda}+x^{\lambda+\alpha_{i}^{m}}+\cdots+ x^{\lambda-(1+(\mathbf{q}%
(\lambda),\alpha_{i}^{m\vee}))\alpha_{i}^{m}},\,\, & \hbox{ if }
(\mathbf{q}(\lambda),\alpha_{i}^{m\vee})<0.
\end{cases}
\end{equation}
The second statement follows from the observation that
\[
P^{m} =\{\lambda\in P \,\, | \,\, \mathbf{q}(\lambda)=\lambda\}. \qedhere
\]

\end{proof}

%%%%%%%%%%%%%%%%%%%%%%%%%%%%%%%%%%%%%%%%%%%%

\begin{rema}
Note that the action of $\overline{\nabla}_{i}$ can alternatively be described
by
\[
\overline{\nabla}_{i}(x^{\lambda})=\frac{x^{\lambda}- x^{s_{i}\lambda
+(\mathbf{r}(\lambda),\alpha_{i}^{m\vee})\alpha_{i}^{m}}} {1-x^{\alpha_{i}%
^{m}}},\qquad\lambda\in P.
\]

\end{rema}

%%%%%%%%%%%%%%%%%%%%%%%%%%%%%%%%%%%%%%%%%%%
Write $\Phi^{m}=\Phi^{m}_{sh}\cup\Phi^{m}_{lg}$ for the division of $\Phi^{m}$
into short and long roots, with the convention $\Phi^{m}=\Phi^{m}_{lg}$ if all
roots have the same length. Write
\[
\textup{size}: \Phi^{m}\rightarrow\{sh,lg\}
\]
for the function on $\Phi^{m}$ satisfying $\textup{size}(\alpha)=\textup{sh}$
iff $\alpha\in\Phi^{m}_{sh}$. Write $\mathbf{k}_{sh}$ and $\mathbf{k}_{lg}$
for the value of $\mathbf{k}$ on $\Phi^{m}_{sh}$ and $\Phi^{m}_{lg}$
respectively.
%%%%%%%%%%%%%%%%%%%%%%%%%%%%%%%%%%%%%%%%%%%%%%

\begin{defi}
[Representation parameters]\label{RepPar2} Let $g_{j}(y)\in\mathbb{C}^{\times
}$ for $j\in\mathbb{Z}$ and $y\in\{sh,lg\}$ be parameters satisfying the
following conditions:

\begin{enumerate}
\item[\textbf{a.}] $g_{j}(y)=-1$ if $j\in n\mathbb{Z}$,

\item[\textbf{b.}] $g_{j}(y)=g_{\mathbf{r}_{n}(j)}(y)$,

\item[\textbf{c.}] $g_{j}(y)g_{n-j}(y)=\mathbf{k}_{y}^{-2}$ if $j\in
\mathbb{Z}\setminus n\mathbb{Z}$.
\end{enumerate}
\end{defi}

%%%%%%%%%%%%%%%%%%%%%%%%%%%%%%%%%%%%%%%%%%%%%%

\begin{rema}
The special case where $g_{j}(y) = g_{j}$, i.e., the parameters do not depend
on root length, was considered in \cite{CG, CG07} and motivated the
generalization above. In the applications considered in those papers, the
$g_{i}$ are taken to be Gauss sums.
\end{rema}

Write $\overline{\lambda}$ for the class of $\lambda\in P$ in the finite
abelian quotient group $P/P^{m}$. By Lemma \ref{relform2},
\begin{equation}
\label{p}\mathbf{p}_{i}(\overline{\lambda}):=-\mathbf{k}_{i}g_{-\mathbf{B}%
(\lambda,\alpha_{i})}(\textup{size}(\alpha_{i}^{m}))
\end{equation}
is a well defined function $\mathbf{p}_{i}: P/P^{m}\rightarrow\mathbb{C}%
^{\times}$ for $i=1,\ldots,r$. Note that $\mathbf{p}_{i}(\overline{\lambda
})=\mathbf{k}_{i}$ if $m(\alpha_{i})\,\,\vert\,\, (\lambda,\alpha_{i}^{\vee})$
by Lemma \ref{relform}\textbf{(a)}. The following theorem is the main result
of this subsection.

%%%%%%%%%%%%%%%%%%%%%%%%%%%%%%%%%%%%%%%%%%%%

\begin{thm}
\label{mainTHM} The formulas
\begin{equation}
\label{formulasTHM}%
\begin{split}
\pi(T_{i})x^{\lambda} & :=(\mathbf{k}_{i}-\mathbf{k}_{i}^{-1})\overline
{\nabla}_{i}(x^{\lambda}) +\mathbf{p}_{i}(\overline{\lambda})x^{s_{i}\lambda
},\\
\pi(x^{\nu})x^{\lambda} & :=x^{\lambda+\nu}%
\end{split}
\end{equation}
for $\lambda\in P$, $i=1,\ldots,r$ and $\nu\in P^{m}$ turn $\mathbb{C}[P]$
into a left $\widetilde{H}^{m}(\mathbf{k})$-module.
\end{thm}

%%%%%%%%%%%%%%%%%%%%%%%%%%%%%%%%%%%%%%%%%%%%

\begin{rema}
\label{repRem} \textbf{(i)} Note that $\mathbb{C}[P^{m}]\subseteq
\mathbb{C}[P]$ is a $\widetilde{H}^{m}(\mathbf{k})$-submodule with respect to
the action \eqref{formulasTHM}. The action on $\mathbb{C}[P^{m}]$ simplifies
to
\[%
\begin{split}
\pi(T_{i})x^{\nu} & =(\mathbf{k}_{i}-\mathbf{k}_{i}^{-1})\nabla_{i}^{m}%
(x^{\nu})+\mathbf{k}_{i}x^{s_{i}\nu},\\
\pi(x^{\mu})x^{\nu} & =x^{\mu+\nu}%
\end{split}
\]
for $i=1,\ldots,r$ and $\mu,\nu\in P^{m}$. It follows that $\mathbb{C}%
[P^{m}]\simeq\widetilde{H}^{m}(\mathbf{k})\otimes_{H(\mathbf{k})}%
\mathbb{C}_{0}$ as $\widetilde{H}^{m}(\mathbf{k})$-modules. In particular for
$m\equiv1$ (which happens for instance when $n=1$), the representation $\pi$
itself is isomorphic to $\widetilde{H}^{m}(\mathbf{k})\otimes_{H(\mathbf{k}%
)}\mathbb{C}_{0} $.\newline\textbf{(ii)} Let $\Lambda$ be a lattice in $E$
satisfying $Q\subseteq\Lambda\subseteq P$. Note that $\Lambda$ is
automatically $W$-invariant. The lattice $\Lambda_{0}:=\Lambda\cap P^{m}$ then
satisfies $Q^{m}\subseteq\Lambda_{0}\subseteq P^{m}$, and
\[
\Lambda_{0}=\{\lambda\in\Lambda\,\, | \,\, \mathbf{B}(\lambda,\alpha
)\equiv0\quad\textup{mod } n\,\,\,\forall\alpha\in\Phi\}
\]
by Lemma \ref{relform2}. Furthermore, $\mathbb{C}[\Lambda]\subseteq
\mathbb{C}[P]$ is a $\widetilde{H}^{m}(\mathbf{k},\Lambda_{0})$-submodule for
the action \eqref{formulasTHM}, with $\widetilde{H}^{m}(\mathbf{k},\Lambda
_{0})$ the subalgebra of $\widetilde{H}^{m}(\mathbf{k})$ generated by
$H(\mathbf{k})$ and $\mathbb{C}[\Lambda_{0}]:= \textup{span}\{x^{\nu}%
\}_{\nu\in\Lambda_{0}}$. We write $\pi_{\Lambda}: \widetilde{H}^{m}%
(\mathbf{k},\Lambda_{0})\rightarrow\textup{End}(\mathbb{C}[\Lambda])$ for the
corresponding representation map.
\end{rema}

%%%%%%%%%%%%%%%%%%%%%%%%%%%%%%%%%%%%%%%%%%%%
The remainder of this subsection is devoted to the proof of Theorem
\ref{mainTHM}. The strategy is to realize the $\widetilde{H}^{m}(\mathbf{k}%
)$-module $(\pi,\mathbb{C}[P])$ as a quotient of the induced $\widetilde{H}%
^{m}(\mathbf{k})$-module
\[
N_{C}:=\widetilde{H}^{m}(\mathbf{k})\otimes_{H(\mathbf{k})}V_{C}%
\]
for an appropriate choice of $W$-invariant subset $0\in C\subseteq P$. Note
that the subrepresentation $N_{\{0\}}$ is isomorphic to $\mathbb{C}[P^{m}]$
viewed as module over $\widetilde{H}^{m}(\mathbf{k})$ by Remark \ref{repRem}%
\textbf{(i)}.

The elements
\[
Y^{\nu}\otimes_{H(\mathbf{k})} v_{\lambda}\qquad(\nu\in P^{m}, \lambda\in C)
\]
form a linear basis of $N_{C}$ and, by the Bernstein-Zelevinsky commutation
relations \eqref{BL1}, the $\widetilde{H}^{m}(\mathbf{k})$-action on $N_{C}$
is explicitly given by
\begin{equation}
\label{actionInduced}%
\begin{split}
T_{i} \bigl(x^{\nu}\otimes_{H(\mathbf{k})}v_{\lambda}\bigr) & = (\mathbf{k}%
_{i}-\mathbf{k}_{i}^{-1})\nabla_{i}^{m}(x^{\nu})\otimes_{H(\mathbf{k}%
)}v_{\lambda}+x^{s_{i}\nu}\otimes_{H(\mathbf{k})}T_{i}v_{\lambda},\\
x^{\mu}\bigl(x^{\nu}\otimes_{H(\mathbf{k})}v_{\lambda}\bigr) & =x^{\mu+\nu
}\otimes_{H(\mathbf{k})}v_{\lambda}%
\end{split}
\end{equation}
for $\lambda\in C$, $i=1,\ldots,r$ and $\mu,\nu\in P^{m}$.

We now continue the analysis of the proof of Theorem \ref{mainTHM} by viewing
the group algebra $\mathbb{C}[P]:=\textup{span}\{x^{\lambda}\}_{\lambda\in P}$
as a free left $\mathbb{C}[P^{m}]$-module by
\[
x^{\nu}\cdot x^{\lambda}:=x^{\lambda+\nu},\qquad\nu\in P^{m},\,\, \lambda\in
P.
\]
This $\mathbb{C}[P^{m}]$-module structure on $\mathbb{C}[P]$ coincides with
the $\mathbb{C}[P^{m}]$-structure that will arise from the desired
$\widetilde{H}^{m}(\mathbf{k})$-action \eqref{formulasTHM} by restriction.

Let $C\subseteq P$ be a $W$-invariant subset and let
\[
\mathbf{c}: C\rightarrow\mathbb{C}^{\times}%
\]
be a (for the moment arbitrary) non-vanishing complex-valued function on $C$.
%%%%%%%%%%%%%%%%%%

\begin{defi}
We write
\[
\psi_{C}^{\mathbf{c}}: N_{C}\rightarrow\mathbb{C}[P]
\]
for the morphism of $\mathbb{C}[P^{m}]$-modules satisfying
\[
\psi_{C}^{\mathbf{c}}(x^{\nu}\otimes_{H(\mathbf{k})}v_{\lambda}):=\mathbf{c}%
(\lambda)^{-1}x^{\lambda+\nu},\qquad\lambda\in C,\,\, \nu\in P^{m}.
\]

\end{defi}

%%%%%%%%%%%%%%%%%%%%%%%%%%%%
We fix from now on the $W$-invariant subset $C\subseteq P$ to be
\begin{equation}
\label{C}C:=\{\lambda\in P \,\, | \,\, |(\lambda,\alpha^{\vee})|\leq
m(\alpha)\quad\forall\, \alpha\in\Phi\}.
\end{equation}
%%%%%%%%%%%%%%%%%%%%%%%%%%%%%%%%%%%%%%%%%%%

\begin{lem}
\label{surj} $\psi_{C}^{\mathbf{c}}: N_{C}\rightarrow\mathbb{C}[P]$ is an
epimorphism of $\mathbb{C}[P^{m}]$-modules.
\end{lem}

%%%%%%%%%%%%%%%%%%%%%%%%%%%%%%%%%%%%%%%%%%%%%%%%

\begin{proof}
We need to show that $\psi_{C}^{\mathbf{c}}$ is surjective. Consider the
action of $\widetilde{W}^{m}=W\ltimes P^{m}$ on $P$ and $E$ by reflections and
translations. Since $C$ is $W$-invariant it suffices to show that each
$\widetilde{W}^{m}$-orbit in $P$ intersects $C$. We show the stronger
statement that each $W\ltimes Q^{m}$-orbit in $P$ intersects $C\cap P^{+}$ in
exactly one point.

Write
\[
E^{+}:=\{v\in E \,\, | \,\, (v,\alpha^{\vee})\geq0\quad\forall\, \alpha\in
\Phi^{+}\}
\]
for the closure of the fundamental Weyl chamber of $E$ with respect to
$\Phi^{+}$. Let $\theta^{m}\in\Phi^{m+}$ be the highest short root with
respect to the base $\{\alpha_{1}^{m},\ldots,\alpha_{r}^{m}\}$ of $\Phi^{m}$.
Then $\theta^{m\vee}\in\Phi^{m\vee+}$ is the highest root of $\Phi^{m\vee}$.

By \cite[§4.3]{Hu} each $W\ltimes Q^{m}$-orbit in $E$ intersects the
fundamental alcove
\[
A_{o}:=\{v\in E^{+} \,\, | \,\, (v,\theta^{m\vee})\leq1\}
\]
in exactly one point. Hence each $W\ltimes Q^{m}$-orbit in $P$ intersects
$A_{o}\cap P$ in exactly one point. Now note that
\[%
\begin{split}
A_{o}\cap P & =\{\lambda\in P^{+} \,\, | \,\, (\lambda,\theta^{m\vee}%
)\leq1\}\\
& =\{\lambda\in P^{+} \,\, | \,\, (\lambda,\alpha^{m\vee})\leq1\quad\forall\,
\alpha\in\Phi^{+}\}\\
& =\{\lambda\in P^{+} \,\, | \,\, (\lambda,\alpha^{\vee})\leq m(\alpha
)\quad\forall\, \alpha\in\Phi^{+}\}\\
& =C\cap P^{+}.\qedhere
\end{split}
\]

\end{proof}

%%%%%%%%%%%%%%%%%%%%%%%%%%%%%%%%%%%%
The map $\psi_{C}^{\mathbf{c}}$ gives rise to an isomorphism
\begin{equation}
\label{isoquot}\overline{\psi}_{C}^{\,\mathbf{c}}: N_{C}/\textup{ker}(\psi
_{C}^{\mathbf{c}}) \overset{\sim}{\longrightarrow}\mathbb{C}[P]
\end{equation}
of $\mathbb{C}[P^{m}]$-modules by Lemma \ref{surj}. We now fine-tune the
normalizing factor $\mathbf{c}$ such that the kernel $\textup{ker}(\psi
_{C}^{\mathbf{c}})\subseteq N_{C}$ is in fact a $\widetilde{H}^{m}%
(\mathbf{k})$-submodule of $N_{C}$. We start with deriving some elementary
properties of the metaplectic divided difference operators $\overline{\nabla
}_{i}$ ($i=1,\ldots,r$).
%%%%%%%%%%%%%%%%%%%%%%%%%%%%%%%%%%%

\begin{lem}
\label{helplemma} Let $i\in\{1,\ldots,r\}$.\newline\textbf{(i)} For
$\lambda\in P$ and $\nu\in P^{m}$ we have
\[
x^{\lambda}\nabla_{i}^{m}(x^{\nu})=\overline{\nabla}_{i}(x^{\lambda+\nu})-
\overline{\nabla}_{i}(x^{\lambda})x^{s_{i}\nu}.
\]
\textbf{(ii)} For $\lambda\in P$ and $\nu\in P^{m}$ we have
\[
x^{\lambda}\nabla_{i}^{m}(x^{\nu}) =
\begin{cases}
\overline{\nabla}_{i}(x^{\lambda+\nu})-x^{\lambda+s_{i}\nu}\qquad &
\hbox{ if }\,\, -m(\alpha_{i})\leq(\lambda,\alpha_{i}^{\vee})<0,\\
\overline{\nabla}_{i}(x^{\lambda+\nu})\qquad & \hbox{ if }\,\, 0\leq
(\lambda,\alpha_{i}^{\vee})<m(\alpha_{i}),\\
\overline{\nabla}_{i}(x^{\lambda+\nu})+x^{s_{i}(\lambda+\nu)}\qquad &
\hbox{ if }\,\, (\lambda,\alpha_{i}^{\vee})=m(\alpha_{i}).
\end{cases}
\]

\end{lem}

%%%%%%%%%%%%%%%%%%%%%%%%%%%%%%%%%%%

\begin{proof}
\textbf{(i)} This follows by a direct computation.\newline\textbf{(ii)} Note
that
\[
(\mathbf{q}(\lambda),\alpha_{i}^{m\vee}) =
\begin{cases}
-1\qquad & \hbox{ if }\,\, -m(\alpha_{i})\leq(\lambda,\alpha_{i}^{\vee})<0,\\
0\qquad & \hbox{ if }\,\, 0\leq(\lambda,\alpha_{i}^{\vee})<m(\alpha_{i}),\\
1\qquad & \hbox{ if }\,\, (\lambda,\alpha_{i}^{\vee})=m(\alpha_{i}),
\end{cases}
\]
hence
\[
\overline{\nabla}_{i}(x^{\lambda}) =
\begin{cases}
x^{\lambda}\qquad & \hbox{ if }\,\, -m(\alpha_{i})\leq(\lambda,\alpha
_{i}^{\vee})<0,\\
0\qquad & \hbox{ if }\,\, 0\leq(\lambda,\alpha_{i}^{\vee})<m(\alpha_{i}),\\
-x^{\lambda-\alpha_{i}^{m}}=-x^{s_{i}\lambda}\qquad & \hbox{ if }\,\,
(\lambda,\alpha_{i}^{\vee})=m(\alpha_{i}).
\end{cases}
\]
Now use \textbf{(i)}.
\end{proof}

%%%%%%%%%%%%%%%%%
The following lemma will play an important role in finding the proper choice
of normalizing factor $\mathbf{c}$.
%%%%%%%%%%%%%%%%%%%%%%%%%%%%%%%%%%%%%%%%%%%%

\begin{lem}
\label{togo} For $\nu\in P^{m}$, $\lambda\in C$ and $i=1,\ldots,r$ we have
\[
\psi_{C}^{\mathbf{c}}\bigl(T_{i}x^{\nu}\otimes_{H(\mathbf{k})}v_{\lambda
}\bigr)= \mathbf{c}(\lambda)^{-1}\bigl((\mathbf{k}_{i}-\mathbf{k}_{i}%
^{-1})\overline{\nabla}_{i}(x^{\lambda+\nu}) +\mathbf{d}_{i}(\lambda
)x^{s_{i}(\lambda+\nu)}\bigr)
\]
with $\mathbf{d}_{i}: C\rightarrow\mathbb{C}^{\times}$ given by
\begin{equation}
\label{secondcst}\mathbf{d}_{i}(\lambda):=
\begin{cases}
\mathbf{c}(\lambda)/\mathbf{c}(s_{i}\lambda)\qquad & \hbox{ if } -m(\alpha
_{i})\leq(\lambda,\alpha_{i}^{\vee})<0,\\
\mathbf{k}_{i}\mathbf{c}(\lambda)/\mathbf{c}(s_{i}\lambda)\qquad & \hbox{ if }
(\lambda,\alpha_{i}^{\vee})=0,\\
\mathbf{c}(\lambda)/\mathbf{c}(s_{i}\lambda)\qquad & \hbox{ if } 0<
(\lambda,\alpha_{i}^{\vee})<m(\alpha_{i}),\\
\mathbf{k}_{i}-\mathbf{k}_{i}^{-1}+\mathbf{c}(\lambda)/\mathbf{c}(s_{i}%
\lambda) \quad & \hbox{ if } (\lambda,\alpha_{i}^{\vee})=m(\alpha_{i}).
\end{cases}
\end{equation}

\end{lem}

%%%%%%%%%%%%%%%%%%%%%%%%%%%%%%%%%%%%%%%%%%%

\begin{proof}
By a direct computation using \eqref{actionInduced}, we have
\begin{equation}
\label{startformula}\psi_{C}(T_{i}x^{\nu}\otimes_{H(\mathbf{k})}v_{\lambda})=
\mathbf{c}(\lambda)^{-1}(\mathbf{k}_{i}-\mathbf{k}_{i}^{-1})x^{\lambda}%
\nabla_{i}^{m}(x^{\nu})+ x^{s_{i}\nu}\psi_{C}(1\otimes_{H_{0}(\mathbf{k}%
)}T_{i}v_{\lambda})
\end{equation}
for $\lambda\in C$ and $\nu\in P^{m}$. We now consider four cases.\newline%
\textbf{Case 1:} $-m(\alpha_{i})\leq(\lambda,\alpha_{i}^{\vee})<0$.\newline
Then
\[
\psi_{C}^{\mathbf{c}}(1\otimes_{H(\mathbf{k})}T_{i}v_{\lambda})=
\mathbf{c}(\lambda)^{-1}(\mathbf{k}_{i}-\mathbf{k}_{i}^{-1})x^{\lambda
}+\mathbf{c}(s_{i}\lambda)^{-1}x^{s_{i}\lambda}.
\]
Substuting into \eqref{startformula} and using Lemma \ref{helplemma} we get
the desired formula
\[
\psi_{C}^{\mathbf{c}}(T_{i}x^{\nu}\otimes_{H(\mathbf{k})}v_{\lambda})=
\mathbf{c}(\lambda)^{-1}\Bigl((\mathbf{k}_{i}-\mathbf{k}_{i}^{-1}%
)\overline{\nabla}_{i}(x^{\lambda+\nu}) +\frac{\mathbf{c}(\lambda)}%
{\mathbf{c}(s_{i}\lambda)}x^{s_{i}(\lambda+\nu)}\Bigr).
\]
\textbf{Case 2:} $(\lambda,\alpha_{i}^{\vee})=0$.\newline Now we have
\[
\psi_{C}^{\mathbf{c}}(1\otimes_{H(\mathbf{k})}T_{i}v_{\lambda})=\mathbf{c}%
(\lambda)^{-1}\mathbf{k}_{i}x^{\lambda} =\mathbf{c}(s_{i}\lambda
)^{-1}\mathbf{k}_{i}x^{s_{i}\lambda}.
\]
Substituting into \eqref{startformula} and using Lemma \ref{helplemma} we now
get the desired formula
\[
\psi_{C}^{\mathbf{c}}(T_{i}x^{\nu}\otimes_{H(\mathbf{k})}v_{\lambda})=
\mathbf{c}(\lambda)^{-1}\Bigl((\mathbf{k}_{i}-\mathbf{k}_{i}^{-1}%
)\overline{\nabla}_{i}(x^{\lambda+\nu}) +\mathbf{k}_{i}\frac{\mathbf{c}%
(\lambda)}{\mathbf{c}(s_{i}\lambda)}x^{s_{i}(\lambda+\nu)}\Bigr).
\]
\textbf{Case 3:} $0<(\lambda,\alpha_{i}^{\vee})<m(\alpha_{i})$.\newline Then
\[
\psi_{C}^{\mathbf{c}}(1\otimes_{H(\mathbf{k})}T_{i}v_{\lambda})=\mathbf{c}%
(s_{i}\lambda)^{-1}x^{s_{i}\lambda}.
\]
Substitution into \eqref{startformula} and using Lemma \ref{helplemma} gives
the desired formula
\[
\psi_{C}^{\mathbf{c}}(T_{i}x^{\nu}\otimes_{H(\mathbf{k})}v_{\lambda})=
\mathbf{c}(\lambda)^{-1}\Bigl((\mathbf{k}_{i}-\mathbf{k}_{i}^{-1}%
)\overline{\nabla}_{i}(x^{\lambda+\nu}) +\frac{\mathbf{c}(\lambda)}%
{\mathbf{c}(s_{i}\lambda)}x^{s_{i}(\lambda+\nu)}\Bigr).
\]
\textbf{Case 4:} $(\lambda,\alpha_{i}^{\vee})=m(\alpha_{i})$.\newline In this
case
\[
\psi_{C}^{\mathbf{c}}(1\otimes_{H(\mathbf{k})}T_{i}v_{\lambda})=\mathbf{c}%
(s_{i}\lambda)^{-1}x^{s_{i}\lambda},
\]
hence substitution into \eqref{startformula} and using Lemma \ref{helplemma}
gives
\[
\psi_{C}^{\mathbf{c}}(T_{i}x^{\nu}\otimes_{H_{0}(\mathbf{k})}v_{\lambda})=
\mathbf{c}(\lambda)^{-1}\Bigl((\mathbf{k}_{i}-\mathbf{k}_{i}^{-1}%
)\overline{\nabla}_{i}(x^{\lambda+\nu}) +\Bigl(\mathbf{k}_{i}-\mathbf{k}%
_{i}^{-1}+\frac{\mathbf{c}(\lambda)}{\mathbf{c}(s_{i}\lambda)}\Bigr)x^{s_{i}%
(\lambda+\nu)}\Bigr),
\]
as desired.
\end{proof}

%%%%%%%%%%%%%%%%%%%%%%%%%%%%%%%%%%%%%%%%%%%%
%%%%%%%%%%%%%%%%%%%%%%%%%%%%%%%%%%%%%%%%%%%%%
We now continue with the proof of Theorem \ref{mainTHM}.
%%%%%%%%%%%%%%%%%%%%%%%%%%%%%%%%%%%%%%%%%
Define parameters $h_{j}(y)\in\mathbb{C}^{\times}$ for $j\in\mathbb{Z}$ and
$y\in\{sh,lg\}$ by
\[%
\begin{split}
h_{j}(y) & :=\mathbf{k}_{y}\qquad\qquad\qquad\hbox{if } j\in n\mathbb{Z}%
_{<0},\\
h_{j}(y) & :=-\mathbf{k}_{y}^{-1}g_{j}(y)^{-1} \qquad\hbox{if } j\in
\mathbb{Z}_{<0}\setminus n\mathbb{Z}_{<0},\\
h_{j}(y) & :=1 \qquad\qquad\quad\qquad\hbox{if } j\in\mathbb{Z}_{\geq0}.
\end{split}
\]
Then $h_{j}(y)=h_{-n+j}(y)$ if $j\in\mathbb{Z}_{<0}$, and $h_{j}%
(y)h_{-sn-j}(y)=1 $ for $j\in\mathbb{Z}_{<0}\setminus n\mathbb{Z}_{<0}$ and
$s\in\mathbb{Z}_{>0}$ such that $-sn<j<0$.

Choose $\mathbf{c}: C\rightarrow\mathbb{C}^{\times}$ by
\begin{equation}
\label{cdef}\mathbf{c}(\lambda):= \prod_{\alpha\in\Phi^{+}}h_{\mathbf{Q}%
(\alpha)(\lambda,\alpha^{\vee})}(\textup{size}(\alpha^{m})), \qquad\lambda\in
C.
\end{equation}
Using
\begin{equation}
\label{crelation}\frac{\mathbf{c}(\lambda)}{\mathbf{c}(s_{i}\lambda)}=
\frac{h_{\mathbf{Q}(\alpha_{i})(\lambda,\alpha_{i}^{\vee})}(\textup{size}%
(\alpha_{i}^{m}))}{h_{-\mathbf{Q}(\alpha_{i})(\lambda,\alpha_{i}^{\vee}%
)}(\textup{size}(\alpha_{i}^{m}))},
\end{equation}
\eqref{relin}, \eqref{secondcst} and Lemma \ref{relform}\textbf{b} one
verifies that for $i=1,\ldots,r$ and $\lambda\in C$,
\[
\mathbf{d}_{i}(\lambda)=
\begin{cases}
h_{-\mathbf{Q}(\alpha_{i})\mathbf{r}_{m(\alpha_{i})}((\lambda,\alpha_{i}%
^{\vee}))}(\textup{size}(\alpha_{i}^{m}))^{-1} \quad & \textup{\ if }\,
m(\alpha_{i})\not \vert \,\, (\lambda,\alpha_{i}^{\vee}),\\
\mathbf{k}_{i}\quad & \textup{\ if }\, m(\alpha_{i})\,\,\,\vert\,\,(\lambda
,\alpha_{i}^{\vee}).
\end{cases}
\]
Rewriting in terms of the representation parameters $g_{j}(y)$ and using Lemma
\ref{relform}\textbf{(b)} we get
\begin{equation}
\label{equalpar}\mathbf{d}_{i}(\lambda)= \mathbf{p}_{i}(\overline{\lambda})
\end{equation}
for $i=1,\ldots,r$ and $\lambda\in C$, with $\mathbf{p}_{i}(\overline{\lambda
}) $ given by \eqref{p}.

Now let $S_{i}: \mathbb{C}[P]\rightarrow\mathbb{C}[P]$ be the linear map
defined by
\[
S_{i}(x^{\lambda}):=(\mathbf{k}_{i}-\mathbf{k}_{i}^{-1})\overline{\nabla}%
_{i}(x^{\lambda})+\mathbf{p}_{i}(\overline{\lambda})x^{s_{i}\lambda}%
,\qquad\lambda\in P,
\]
then Lemma \ref{togo} and \eqref{equalpar} show that for $i=1,\ldots,r$ and
$\lambda\in C$, $\nu\in P^{m}$,
\begin{equation}
\label{relint}S_{i}\bigl(\psi_{C}^{\mathbf{c}}(x^{\nu}\otimes_{H(\mathbf{k}%
)}v_{\lambda})\bigr)= \psi_{C}^{\mathbf{c}}\bigl(T_{i}x^{\nu}\otimes
_{H(\mathbf{k})}v_{\lambda}\bigr).
\end{equation}
Hence the kernel of the epimorphism $\psi_{C}^{\mathbf{c}}: N_{C}%
\twoheadrightarrow\mathbb{C}[P]$ is a $\widetilde{H}^{m}(\mathbf{k}%
)$-submodule. By \eqref{relint} it follows that the $\widetilde{H}%
^{m}(\mathbf{k})$-module structure on $\mathbb{C}[P]$, inherited from the
quotient $\widetilde{H}^{m}(\mathbf{k})$-module $N_{C}/\textup{ker}(\psi
_{C}^{\mathbf{c}})$ by the $\mathbb{C}[P^{m}]$-module isomorphism
$\overline{\psi}_{C}^{\,\mathbf{c}} $ (see \eqref{isoquot}), is explicitly
given by \eqref{formulasTHM}. This completes the proof of Theorem
\ref{mainTHM}.

In subsequent sections, we will work with some conjugations of $\pi$, so the
following Lemma will be useful.

%%%%%%%%%%%%%%%%%%%%%%%%%%%%%%%%

\begin{lem}
\label{pi_conj_stab} Let $\Lambda$ be a lattice in $E$ satisfying $Q
\subseteq\Lambda\subseteq P$. Let $h \in\widetilde{H}^{m}(\mathbf{k},
\Lambda_{0})$ and $\mu\in P$. Then $x^{-\mu} \pi(h) x^{\mu}$ preserves
$\mathbb{C}[\Lambda]$.
\end{lem}

%%%%%%%%%%%%%%%%%%%%%%%%%%%%%%%%%

\begin{proof}
Since $\pi(x^{\nu})$ for $\nu\in\Lambda_{0}$ commutes with multiplication by
$x^{\mu}$, we need only check that $x^{-\mu} \pi(T_{i}) x^{\mu}$ preserves
$\mathbb{C}[\Lambda]$ for $1 \leq i \leq r$. Let $\lambda\in\Lambda$. By
Theorem \ref{mainTHM}, we have
\[
x^{-\mu} \pi(T_{i})(x^{\lambda+ \mu}) = x^{-\mu}(\mathbf{k}_{i} -
\mathbf{k}_{i}^{-1})\overline{\nabla}_{i}(x^{\lambda+ \mu}) + x^{-\mu}
\mathbf{p}_{i}(\overline{\lambda+ \mu}) x^{s_{i}(\lambda+ \mu)}.
\]
We have
\[
x^{-\mu+ s_{i}(\lambda+ \mu)} = x^{\lambda- ( \lambda+ \mu, \alpha_{i}^{\vee}
) \alpha_{i}} \in\mathbb{C}[\Lambda],
\]
since $( \lambda+ \mu, \alpha_{i}^{\vee} ) \alpha_{i} \in Q$. For the other
term, by \eqref{nablabar} and Lemma \ref{Proj}, we have
\[
\overline{\nabla}_{i} (x^{\lambda+ \mu}) = x^{\lambda+ \mu}g,
\]
where $g \in\mathbb{C}[Q^{m}]$. So now $x^{-\mu} \overline{\nabla}%
_{i}(x^{\lambda+ \mu}) \in\mathbb{C}[\Lambda]$.
\end{proof}

%%%%%%%%%%%%%%%%%%%%%%%%%%%%%%%%%%%%%%%%%%%%

%%%%%%%%%%%%%%%%%%%%%%%%%%%%%%%%%%%%%%

\subsection{The metaplectic Weyl group representation}

\label{subs33}
%%%%%%%%%%%%%%%%%%%%%%%%%%%%%%%%%%%%%%%

Let $\mathbb{C}(P^{m})$ be the quotient field of $\mathbb{C}[P^{m}]$.

Let $\widetilde{H}^{m}_{\text{loc}}(\mathbf{k})$ be algebra obtained by
localizing the extended affine Hecke algebra $\widetilde{H}^{m}(\mathbf{k})$
at the multiplicative subset $\mathbb{C}[P^{m}]\setminus\{0\}$ (which
satisfies the right Ore condition). The canonical algebra embedding
$\mathbb{C}[P^{m}]\hookrightarrow\widetilde{H}^{m}_{\text{loc}}(\mathbf{k})$
uniquely extends to an algebra embedding $\mathbb{C}(P^{m})\hookrightarrow
\widetilde{H}^{m}_{\text{loc}}(\mathbf{k})$, with $\mathbb{C}(P^{m})$ the
quotient field of $\mathbb{C}[P^{m}]$. Furthermore,
\begin{equation}
\label{lineariso}H(\mathbf{k})\otimes\mathbb{C}(P^{m})\overset{\sim
}{\longrightarrow}\widetilde{H}^{m}_{\text{loc}}(\mathbf{k})
\end{equation}
as vector spaces by the multiplication map. The defining relations of
$\widetilde{H}^{m}_{\text{loc}}(\mathbf{k})$ with respect to the decomposition
\eqref{lineariso} are captured by the extended cross relations
\[
T_{i}f=(s_{i}f)T_{i}+\bigl(\mathbf{k}_{i}-\mathbf{k}_{i}^{-1}\bigr)\left(
\frac{f-s_{i}f} {1-x^{\alpha_{i}^{m}}}\right)
\]
for $i\in\{1,\ldots,r\}$ and $f\in\mathbb{C}(P^{m})$, where we use the
extension of the $W$-action on $\mathbb{C}[P^{m}]$ to $\mathbb{C}(P^{m})$ by
field automorphisms.

If the multiplicity function $\mathbf{k}$ is identically equal to one then
$\widetilde{H}^{m}_{\text{loc}}(\mathbf{k})$ is isomorphic to the semi-direct
product algebra
\[
W\ltimes\mathbb{C}(P^{m}):=\mathbb{C}[W]\otimes\mathbb{C}(P^{m})
\]
with algebra structure given by $(v\otimes f)(w\otimes g):=vw\otimes
(w^{-1}f)g$ for $v,w\in W$ and $f,g\in\mathbb{C}(P^{m})$. We write $gw$ for
the element $(1\otimes g)(w\otimes1)=w\otimes w^{-1}g$ in $W\ltimes
\mathbb{C}(P^{m})$ if no confusion can arise.

Define for $\alpha\in\Phi$ the $c$-functions $c_{\alpha}=c_{\alpha}^{m}%
\in\mathbb{C}(Q^{m})$ by
\begin{equation}
\label{cfunction}c_{\alpha}:=\frac{1-\mathbf{k}_{\alpha}^{2}x^{\alpha^{m}}%
}{1-x^{\alpha^{m}}}.
\end{equation}
We write $c_{i}:=c_{\alpha_{i}}$ ($i=1,\ldots,r$) for the $c$-functions at the
simple roots. Note that $w(c_{\alpha})=c_{w\alpha}$ for $w\in W$ and
$\alpha\in\Phi$.

By \cite{Ka} we have the following result.
%%%%%%%%%%%%%%%%%%%%%%%%%%%%%

\begin{thm}
\label{localization_thm} There exists a unique algebra isomorphism
\begin{equation}
\label{phi}\varphi: W\ltimes\mathbb{C}(P^{m})\overset{\sim}{\longrightarrow
}\widetilde{H}^{m}_{\text{loc}}(\mathbf{k})
\end{equation}
given by $\varphi(f)=f$ for $f\in\mathbb{C}(P^{m})$ and
\begin{equation}
\label{varphi}\varphi(s_{i}):=\frac{\mathbf{k}_{i}}{c_{i}}T_{i}+1-\frac
{\mathbf{k}_{i}^{2}}{c_{i}}%
\end{equation}
for $i=1\ldots,r$.
\end{thm}

%%%%%%%%%%%%%%%%%%%%%%%%%%%%%%
The $\varphi(s_{i})$ are the so-called \textit{normalized intertwiners} of the
extended affine Hecke algebra $\widetilde{H}^{m}(\mathbf{k})$ (see \cite{Ka}
and, e.g., \cite[§3.3.3]{Ch}). They play an instrumental role in the
representation theory of $\widetilde{H}^{m}(\mathbf{k})$.

Note that for $i=1\ldots,r$ we have
\[
\varphi^{-1}(T_{i})=\mathbf{k}_{i}+\mathbf{k}_{i}^{-1}c_{i}(s_{i}-1)
\]
in $W\ltimes\mathbb{C}(P^{m})$, which are the Demazure-Lusztig operators
\cite{Lu}.
%%%%%%%%%%%%%%%%%%%%%%%%%%%%%%%%

\begin{rema}
The localization isomorphism \eqref{phi} extends to the double affine Hecke
algebra, see \cite[§3.3.3]{Ch}. In its most natural form it involves
normalized intertwiners dual to $\varphi(s_{i})$, as well as an additional
dual intertwiner naturally attached to the simple simple reflection of the
affine Weyl group $W\ltimes Q^{m}$.
\end{rema}

%%%%%%%%%%%%%%%%%%%%%%%%%%%%%%%%%%%%

%%%%%%%%%%%%%%%%%%%%%%%%%%%%%%%%%

\begin{defi}
\label{BaxterizedRep} Let $(\rho,M)$ be a left $\widetilde{H}^{m}(\mathbf{k}%
)$-module. Write $(\rho_{\text{loc}},M_{\text{loc}})$ for the associated
localized $W\ltimes\mathbb{C}(P^{m})$-module
\[
M_{\text{loc}}:=\widetilde{H}^{m}_{\text{loc}}(\mathbf{k})\otimes
_{\widetilde{H}^{m}(\mathbf{k})}M
\]
with representation map $\rho_{\text{loc}}: W\ltimes\mathbb{C}(P^{m}%
)\rightarrow\textup{End}(M_{\text{loc}})$ defined by
\[
\rho_{\text{loc}}(X)(h\otimes_{\widetilde{H}^{m}(\mathbf{k})}m):=(\varphi
(X)h)\otimes_{\widetilde{H}^{m}(\mathbf{k})}m
\]
for $X\in W\ltimes\mathbb{C}(P^{m})$, $h\in\widetilde{H}^{m}_{\text{loc}%
}(\mathbf{k})$ and $m\in M$.
\end{defi}

%%%%%%%%%%%%%%%%%%%%%%%%%%%%%%%%%
Note that $M_{\text{loc}}\simeq\mathbb{C}(P^{m})\otimes_{\mathbb{C}[P^{m}]}M$
as vector spaces with the isomorphism mapping $fh\otimes_{\widetilde{H}%
^{m}(\mathbf{k})}m$ to $f\otimes_{\mathbb{C}[P^{m}]}\rho(h)m$ for
$f\in\mathbb{C}(P^{m})$, $h\in H(\mathbf{k})$ and $m\in M$ (the map is well
defined by the Bernstein-Zelevinsky presentation of $\widetilde{H}%
^{m}_{\text{loc}}(\mathbf{k})$).
%%%%%%%%%%%%%%%%%%%%%%%%%%%%%%%%%%

\begin{rema}
\label{backtorho} Identifying $M$ as subspace of $M_{\text{loc}}$ by the
linear embedding $M\hookrightarrow M_{\text{loc}}$, $m\mapsto1\otimes
_{\widetilde{H}^{m}(\mathbf{k})}m$, we have
\[
\rho_{\text{loc}}(\varphi^{-1}(h))m=\rho(h)m,\qquad h\in\widetilde{H}%
^{m}(\mathbf{k}),\,\, m\in M.
\]

\end{rema}

%%%%%%%%%%%%%%%%%%%%%%%%%%%%%%%%%%%

\begin{rema}
A Bethe integrable system with extended affine Hecke algebra symmetry is a
$\widetilde{H}^{m}(\mathbf{k})$-module $V$ endowed with the integrable
structure obtained from the action of the associated dual intertwiners on
$\mathbb{C}(P^{m})\otimes V$. The integrable structure is thus encoded by
solutions of (braid version of) generalized quantum Yang-Baxter equations with
spectral parameter. In the literature on integrable systems one sometimes says
that the integrable structure arises from Baxterizing the affine Hecke algebra
module structure on the quantum state space. See e.g. \cite{S} for an example
involving the Heisenberg XXZ spin-$\frac{1}{2}$ chain.

The intertwiners are also instrumental in the construction of the quantum
affine KZ equations, see, e.g., \cite[§1.3.2]{Ch}.
\end{rema}

%%%%%%%%%%%%%%%%%%%%%%%%%%%%%%%%

Recall the metaplectic affine Hecke algebra representation $(\pi
,\mathbb{C}[P])$ from Theorem \ref{mainTHM}. In the following proposition we
explicitly describe $(\pi_{\text{loc}},\mathbb{C}[P]_{\text{loc}})$.
%%%%%%%%%%%%%%%%%%%%%%%%%%%%%%%%%%%%%%

\begin{prop}
\label{locrep} \textbf{(i)} $\mathbb{C}(P)=\bigoplus_{\overline{\lambda}\in
P/P^{m}}\mathbb{C}(P^{m})x^{\lambda}$.\newline\textbf{(ii)} $\mathbb{C}%
[P]_{\text{loc}}\simeq\mathbb{C}(P)$ as vector spaces by
\[
fh\otimes_{\widetilde{H}^{m}(\mathbf{k})}g\mapsto f(\pi(h)g),\qquad
f\in\mathbb{C}(P^{m}), \,\,\, g\in\mathbb{C}[P],\,\, h\in H(\mathbf{k}).
\]
\textbf{(iii)} The $\pi_{\text{loc}}$-action of $W\ltimes\mathbb{C}(P^{m})$ on
$\mathbb{C}(P)$ (identifying $\mathbb{C}[P]_{\text{loc}}$ with $\mathbb{C}(P)$
using the linear isomorphism from \textbf{(ii)}) is explicitly given by
\[%
\begin{split}
\pi_{\text{loc}}(s_{i})(fx^{\lambda}) & = (s_{i}f)\left( \left( \frac
{(1-\mathbf{k}_{i}^{2})x^{-(\mathbf{q}(\lambda),\alpha_{i}^{m\vee})\alpha
_{i}^{m}}} {1-\mathbf{k}_{i}^{2}x^{\alpha_{i}^{m}}}\right) x^{\lambda}+ \left(
\frac{\mathbf{k}_{i}\mathbf{p}_{i}(\overline{\lambda})(1-x^{\alpha_{i}^{m}}%
)}{1-\mathbf{k}_{i}^{2}x^{\alpha_{i}^{m}}}\right) x^{s_{i}\lambda}\right) ,\\
\pi_{\text{loc}}(g)(fx^{\lambda}) & =gf x^{\lambda}%
\end{split}
\]
for $f,g\in\mathbb{C}(P^{m})$, $\lambda\in P$ and $i=1,\ldots,r$ (recall that
$\mathbf{p}_{i}(\overline{\lambda})$ is given by \eqref{p}).
\end{prop}

%%%%%%%%%%%%%%%%%%%%%%%%%%%%%%%%%%%%%%

\begin{proof}
\textbf{(i)} Let $G$ be the group of characters of the finite abelian group
$P/P^{m}$. It acts by field automorphisms on $\mathbb{C}(P)$ by
\[
\chi\cdot x^{\lambda}:=\chi(\overline{\lambda})x^{\lambda},\qquad\lambda\in
P,\,\, \chi\in G.
\]
Decomposing $\mathbb{C}(P)$ in $G$-isotypical components yields
\[
\mathbb{C}(P)=\bigoplus_{\overline{\lambda}\in P/P^{m}}\mathbb{C}%
(P)^{G}x^{\lambda},
\]
with $\mathbb{C}(P)^{G}$ the subfield of $G$-invariant elements in
$\mathbb{C}(P)$. It remains to show that $\mathbb{C}(P)^{G}=\mathbb{C}(P^{m}%
)$, for which it suffices to show that $\mathbb{C}[P]^{G}=\mathbb{C}[P^{m}]$.
The latter follows from the fact that
\[
\textup{pr}(x^{\lambda})=\delta_{\overline{\lambda},\overline{0}}x^{\lambda
},\qquad\lambda\in P
\]
for the projection map $\textup{pr}: \mathbb{C}[P]\twoheadrightarrow
\mathbb{C}[P]^{G}$ defined by
\[
\textup{pr}(f):=\frac{1}{\#G}\sum_{\chi\in G}\chi\cdot f,\qquad f\in
\mathbb{C}[P].
\]
\textbf{(ii)} We have
\[%
\begin{split}
\mathbb{C}[P]_{\text{loc}} & =\widetilde{H}^{m}_{\text{loc}}(\mathbf{k})
\otimes_{\widetilde{H}^{m}(\mathbf{k})}\mathbb{C}[P]\\
& \simeq\mathbb{C}(P^{m})\otimes_{\mathbb{C}[P^{m}]}\mathbb{C}[P]\simeq
\mathbb{C}(P)
\end{split}
\]
with the last isomorphism mapping $f\otimes_{\mathbb{C}[P^{m}]}g$ to $fg$ for
$f\in\mathbb{C}(P^{m})$ and $g\in\mathbb{C}[P]$. This is well defined and an
isomorphism due to the second formula of \eqref{formulasTHM} and due to part
\textbf{(i)} of the proposition. The result now immediately follows.\newline%
\textbf{(iii)} For $f,g\in\mathbb{C}(P^{m})$ and $\lambda\in P$ we have
\[%
\begin{split}
\pi_{\text{loc}}(g)(fx^{\lambda}) & =\pi_{\textup{loc}}(g)\bigl(f\otimes
_{\widetilde{H}^{m}(\mathbf{k})} x^{\lambda}\bigr)\\
& =(gf)\otimes_{\widetilde{H}^{m}(\mathbf{k})}x^{\lambda}=gfx^{\lambda}%
=\pi_{\text{loc}}(gf)x^{\lambda},
\end{split}
\]
this establishes the second formula. For the first formula it then suffices to
prove that
\begin{equation}
\label{sidot}\pi_{\text{loc}}(s_{i})(x^{\lambda})= \left( \frac{(1-\mathbf{k}%
_{i}^{2})x^{-(\mathbf{q}(\lambda),\alpha_{i}^{m\vee})\alpha_{i}^{m}}}
{1-\mathbf{k}_{i}^{2}x^{\alpha_{i}^{m}}}\right) x^{\lambda}+ \left(
\frac{\mathbf{k}_{i}\mathbf{p}_{i}(\overline{\lambda})(1-x^{\alpha_{i}^{m}}%
)}{1-\mathbf{k}_{i}^{2}x^{\alpha_{i}^{m}}}\right) x^{s_{i}\lambda}%
\end{equation}
for $i=1,\ldots,r$ and $\lambda\in P$. By the first formula of
\eqref{formulasTHM} we have
\[%
\begin{split}
\pi_{\text{loc}}(s_{i})x^{\lambda} & = \frac{\mathbf{k}_{i}}{c_{i}}\pi
(T_{i})x^{\lambda}+\Bigl(1-\frac{\mathbf{k}_{i}^{2}}{c_{i}}\Bigr)x^{\lambda}\\
& =\frac{\mathbf{k}_{i}}{c_{i}}\Bigl((\mathbf{k}_{i}-\mathbf{k}_{i}%
^{-1})\Bigl(\frac{x^{\lambda}-x^{\lambda-(\mathbf{q}(\lambda),\alpha
_{i}^{m\vee})\alpha_{i}^{m}}} {1-x^{\alpha_{i}^{m}}}\Bigr)+\mathbf{p}%
_{i}(\overline{\lambda})x^{s_{i}\lambda}\Bigr) +\Bigl(1-\frac{\mathbf{k}%
_{i}^{2}}{c_{i}}\Bigr)x^{\lambda}.
\end{split}
\]
Substituting the definition of the $c$-function $c_{i}$ (see
\eqref{cfunction}) gives
\[
\pi_{\text{loc}}(s_{i})x^{\lambda}= \frac{ (\mathbf{k}_{i}^{2}-1)(x^{\lambda
}-x^{\lambda-(\mathbf{q}(\lambda),\alpha_{i}^{m\vee})\alpha_{i}^{m}})
+\mathbf{k}_{i}\mathbf{p}_{i}(\overline{\lambda})(1-x^{\alpha_{i}^{m}%
})x^{s_{i}\lambda}+(1-\mathbf{k}_{i}^{2})x^{\lambda}} {1-\mathbf{k}_{i}%
^{2}x^{\alpha_{i}^{m}}}.
\]
Simplifying the expression gives \eqref{sidot}.
\end{proof}

%%%%%%%%%%%%%%%%%%%%%%%%%%%%%%%%%%%%%%%%%%%%%

\begin{rema}
\label{standard_W} Since $q(0)=0$ and $\mathbf{p}_{i}(\overline{0}%
)=\mathbf{k}_{i}$ we have $\pi_{\text{loc}}(s_{i})1=1$. Hence $\mathbb{C}%
(P^{m})$ is a $\pi_{\text{loc}}$-submodule of $\mathbb{C}(P)$ with the
$W\ltimes\mathbb{C}(P^{m})$-action reducing to the standard one,
\[
\pi_{\text{loc}}(s_{i})f=s_{i}f,\qquad\pi_{\text{loc}}(g)f:=gf
\]
for $i=1,\ldots,r$ and $f,g\in\mathbb{C}(P^{m})$.
\end{rema}

%%%%%%%%%%%%%%%%%%%%%%%%%%%%%%%%%%%%%%%%%%%%%%
%%%%%%%%%%%%%%%%%%%%%%%%%%%%%%%%%%%%%%%%%%%%%%%
Recall the definition of the representation parameters $g_{j}(y)$
($j\in\mathbb{Z}$, $y\in\{sh,lg\}$), see Definition \ref{RepPar2}. We
conjugate the $\pi_{\text{loc}}$-action by a certain factor, in order to line
it up with the Weyl group action of Chinta-Gunnells \cite{CG07, CG}.
%%%%%%%%%%%%%%%%%%%%%%%%%%%%%%%%%%%%%%%%%%%%%%

\begin{thm}
[Metaplectic Weyl group representation]\label{mWrep} The following formulas
turn $\mathbb{C}(P)$ into a left $W\ltimes\mathbb{C}(P^{m})$-module,
\begin{equation}
\label{formulaCGPgen}%
\begin{split}
\sigma(s_{i})(fx^{\lambda}):= & \frac{(1-\mathbf{k}_{i}^{2})x^{(\mathbf{q}%
(-\lambda),\alpha_{i}^{m\vee})\alpha_{i}^{m}}} {(1-\mathbf{k}_{i}^{2}%
x^{\alpha_{i}^{m}})}(s_{i}f)x^{\lambda}\\
& \,\,\,+\mathbf{k}_{i}^{2}g_{\mathbf{Q}(\alpha_{i})-\mathbf{B}(\lambda
,\alpha_{i})}(\textup{size}(\alpha_{i}^{m})) \frac{(1-x^{-\alpha_{i}^{m}}%
)}{(1-\mathbf{k}_{i}^{2}x^{\alpha_{i}^{m}})}(s_{i}f)x^{\alpha_{i}+s_{i}%
\lambda},\\
\sigma(g)(fx^{\lambda}):= & gfx^{\lambda}%
\end{split}
\end{equation}
for $f,g\in\mathbb{C}(P^{m})$, $\lambda\in P$ and $i=1,\ldots,r$.
\end{thm}

%%%%%%%%%%%%%%%%%%%%%%%%%%%%%%%%%%%%%%%%%%%%
%%%%%%%%%%%%%%%%%%%%%%%%%%%%%%%%%%%%%%%%%%%%

\begin{proof}
Write $\rho:=\frac{1}{2}\sum_{\alpha\in\Phi^{+}}\alpha$ and $\rho^{m}%
:=\frac{1}{2}\sum_{\alpha\in\Phi^{+}}\alpha^{m}$ for the half sum of positive
roots of $\Phi$ and $\Phi^{m}$ respectively. Then $s_{i}(\rho)=\rho-\alpha
_{i}$ and $s_{i}(\rho^{m})=\rho^{m}-\alpha_{i}^{m}$, in particular $\rho
=\sum_{i=1}^{r}\varpi_{i}\in P $ and
\[
\rho^{m}=\sum_{i=1}^{r}\varpi_{i}^{m}=\sum_{i=1}^{r}m(\alpha_{i})\varpi_{i}\in
P^{m}.
\]
Consider now the $W\ltimes\mathbb{C}(P^{m})$ on $\mathbb{C}(P)$ defined by
\begin{equation}
\label{sigma_piloc}\sigma(X)f:=x^{\rho-\rho^{m}}\pi_{\text{loc}}%
(X)(x^{\rho^{m}-\rho}f),\qquad X\in W\ltimes\mathbb{C}(P^{m}),\,\,
f\in\mathbb{C}(P).
\end{equation}
Then $\sigma(g)f=gf$ for $g\in\mathbb{C}(P^{m})$ and $f\in\mathbb{C}(P)$, and
\begin{equation}
\label{almost}%
\begin{split}
\sigma(s_{i})x^{\lambda} & =\frac{(1-\mathbf{k}_{i}^{2})x^{-(\mathbf{q}%
(\lambda+\rho^{m}-\rho), \alpha_{i}^{m\vee})\alpha_{i}^{m}}}{1-\mathbf{k}%
_{i}^{2}x^{\alpha_{i}^{m}}}x^{\lambda}\\
& \qquad\qquad-\left( \frac{\mathbf{k}_{i}\mathbf{p}_{i}(\overline
{\lambda-\rho})(1-x^{-\alpha_{i}^{m}})} {1-\mathbf{k}_{i}^{2}x^{\alpha_{i}%
^{m}}}\right) x^{\alpha_{i}+s_{i}\lambda}.
\end{split}
\end{equation}
Note that
\[
\mathbf{r}(\lambda+\rho^{m}-\rho)=\rho^{m}-\rho-\mathbf{r}(-\lambda)
\]
since $r_{s}(t+s-1)=s-1-r_{s}(-t)$ for $s\in\mathbb{Z}_{>0}$ and
$t\in\mathbb{Z}$, hence
\[
\mathbf{q}(\lambda+\rho^{m}-\rho)=-\mathbf{q}(-\lambda).
\]
Furthermore,
\[
-\mathbf{k}_{i}\mathbf{p}_{i}(\overline{\lambda-\rho})=\mathbf{k}_{i}%
^{2}g_{\mathbf{Q}(\alpha_{i})-\mathbf{B}(\lambda,\alpha_{i})}(\textup{size}%
(\alpha_{i}^{m}))
\]
for $\lambda\in P$ since $-\mathbf{B}(\lambda-\rho,\alpha_{i})= \mathbf{Q}%
(\alpha_{i})-\mathbf{B}(\lambda,\alpha_{i})$. Substituting these two formulas
in \eqref{almost} gives the desired result.
\end{proof}

%%%%%%%%%%%%%%%%%%%%%%%%%%%%%

As in Remark \ref{repRem}\textbf{(ii)}, fix a lattice $\Lambda\subseteq E$
satisfying $Q\subseteq\Lambda\subseteq P$ and set $\Lambda_{0}:=\Lambda\cap
P^{m} $. Then $Q^{m}\subseteq\Lambda_{0}\subseteq P^{m}$ and recall that
$\Lambda_{0}$ can alternatively be described as
\[
\Lambda_{0}=\{\lambda\in\Lambda\,\, | \,\, \mathbf{B}(\lambda,\alpha
)\equiv0\quad\textup{\ mod } n \,\,\,\forall\alpha\in\Phi\},
\]
which places us directly in the context of \cite{CG}. Note that $\Lambda$ and
$\Lambda_{0}$ are automatically $W$-stable. In particular the subalgebra of
$W\ltimes\mathbb{C}(P^{m})$ generated by $W$ and $\mathbb{C}(\Lambda_{0})$ is
isomorphic to the semi-direct product algebra $W\ltimes\mathbb{C}(\Lambda
_{0})$.

Let $\mathbb{C}(\Lambda_{0})$ and $\mathbb{C}(\Lambda)$ be the subfields of
$\mathbb{C}(P)$ generated by $x^{\nu}$ ($\nu\in\Lambda_{0})$ and $x^{\lambda}$
($\lambda\in\Lambda$) respectively. Similarly to Proposition \ref{locrep}%
\textbf{(i)} we have the decomposition
\[
\mathbb{C}(\Lambda)=\bigoplus_{\overline{\lambda}\in\Lambda/\Lambda_{0}}
\mathbb{C}(\Lambda_{0})x^{\lambda}.
\]
Then $\mathbb{C}(\Lambda)\subseteq\mathbb{C}(P)$ is a $W\ltimes\mathbb{C}%
(\Lambda_{0})$-submodule with respect to the action $\sigma$. Writing
\[
\sigma_{\Lambda}: W\ltimes\mathbb{C}(\Lambda_{0})\rightarrow\textup{End}%
(\mathbb{C}(\Lambda))
\]
for the resulting representation map, we get
%%%%%%%%%%%%%%%%%%%%%%%%%%%%%%%%%%%%%%%

\begin{cor}
\label{CGcor} In the setup as above, the representation map $\sigma_{\Lambda}$
is explicitly given by
\begin{equation}
\label{formulaCGPgen2}%
\begin{split}
\sigma_{\Lambda}(s_{i})(fx^{\lambda}):= & \frac{(1-\mathbf{k}_{i}%
^{2})x^{(\mathbf{q}(-\lambda),\alpha_{i}^{m\vee})\alpha_{i}^{m}}}
{(1-\mathbf{k}_{i}^{2}x^{\alpha_{i}^{m}})}(s_{i}f)x^{\lambda}\\
& \,\,\,+\mathbf{k}_{i}^{2}g_{\mathbf{Q}(\alpha_{i})-\mathbf{B}(\lambda
,\alpha_{i})}(\textup{size}(\alpha_{i}^{m})) \frac{(1-x^{-\alpha_{i}^{m}}%
)}{(1-\mathbf{k}_{i}^{2}x^{\alpha_{i}^{m}})}(s_{i}f)x^{\alpha_{i}+s_{i}%
\lambda},\\
\sigma(g)(fx^{\lambda}):= & gfx^{\lambda}%
\end{split}
\end{equation}
for $f,g\in\mathbb{C}(\Lambda_{0})$, $\lambda\in\Lambda$ and $i=1,\ldots,r$.
\end{cor}

%%%%%%%%%%%%%%%%%%%%%%%%%%%%%

\begin{rema}
Consider the special case that $\mathbf{k}: \Phi^{m}\rightarrow\mathbb{C}%
^{\times}$ is constant and the representation parameters $g_{j}(y)$ satisfy
$g_{j}(sh)=g_{j}(lg)$ for all $j\in\mathbf{Z}$. We call this the equal Hecke
and representation parameter case. Then $\sigma_{\Lambda}$ is exactly the
Chinta-Gunnells \cite{CG07,CG} Weyl group action. This is immediately apparent
by comparing \eqref{formulaCGPgen2} with \cite[(7)]{CGP} (the parameter $v$ in
\cite{CGP} corresponds to $\mathbf{k}^{2}$). Note that our technique gives an
independent and uniform proof that the formulas of Chinta-Gunnells do indeed
give an action of the Weyl group.
\end{rema}

%%%%%%%%%%%%%%%%%%%%%%%%%%%%%

\begin{rema}
Note that $\sigma_{\Lambda}$ reduces at $n=1$ to the standard $W$-action.
However, it is in fact not the standard action on $\mathbb{C}(P^{m})$, due to
the fact that we have conjugated $\pi_{\text{loc}}$ by $x^{\rho- \rho^{m}}$
(compare with Remark \ref{standard_W}).
\end{rema}

%%%%%%%%%%%%%%%%%%%%%%%%%%%%%

Set $\Phi(w):=\Phi^{+}\cap w^{-1}\Phi^{-}$ ($w\in W$) and let $w_{0}\in W$ be
the longest Weyl group element.
%%%%%%%%%%%%%%%%%%%%%%%%%%%%%%%

\begin{defi}
For $\lambda\in P^{+}$ define $\widetilde{\mathcal{W}}_{\lambda}\in
\mathbb{C}(P)$ by
\[
\widetilde{\mathcal{W}}_{\lambda}:=\Bigl(\prod_{\alpha\in\Phi^{+}}c_{\alpha
}\Bigr) \sum_{w\in W}(-1)^{\ell(w)}\Bigl(\prod_{\alpha\in\Phi(w^{-1}%
)}x^{\alpha^{m}}\Bigr) \sigma(w)\bigl(x^{w_{0}\lambda}\bigr).
\]

\end{defi}

%%%%%%%%%%%%%%%%%%%%%%%%%%%%%%%%%

In the equal Hecke and parameter case, McNamara's \cite[Thm. 15.2]{McN}
metaplectic Casselman-Shalika formula relates $\widetilde{\mathcal{W}%
}_{\lambda}$ to the spherical Whittaker function of metaplectic covers of
unramified reductive groups over local fields, see also \cite[Thm. 16]{CGP}.
It is a natural open problem what the corresponding representation theoretic
interpretation is of $\widetilde{\mathcal{W}}_{\lambda}$ in the unequal Hecke
and/or representation parameter case.

In the following section we will obtain in Theorem \ref{WhittakerDL} an
expression of $\widetilde{\mathcal{W}}_{\lambda}$ in terms of metaplectic
analogues of Demazure-Lusztig operators, generalizing \cite[Thm. 16]{CGP}.

%%%%%%%%%%%%%%%%%%%%%%%%%%%%%%%%%%%%%%%%%%%%%%%

%%%%%%%%%%%%%%%%%%%%%%%%%%%%%%%%%%%%%%

\section{Metaplectic Demazure-Lusztig operators}

\label{sect4}
%%%%%%%%%%%%%%%%%%%%%%%%%%%%%%%%%%%%%%
In the previous section we used the localization isomorphism $\varphi:
W\ltimes\mathbb{C}(P^{m})\overset{\sim}{\longrightarrow} \widetilde{H}%
^{m}_{\text{loc}}(\mathbf{k})$ to obtain the metaplectic Weyl group
representation $\sigma$ from the metaplectic affine Hecke algebra
representation $\pi$. In this section we use the localization isomorphism to
turn the metaplectic Weyl group representation $\sigma$ into a localized
affine Hecke algebra representation involving metaplectic Demazure-Lusztig
type operators. This leads to a generalization of some of the results in
\cite[§3]{CGP} to unequal Hecke and representation parameters, and simplifies
some of the proofs in \cite[§3]{CGP}.

Define the algebra map
\[
\tau: \widetilde{H}^{m}_{\text{loc}}(\mathbf{k})\rightarrow\textup{End}%
(\mathbb{C}(P))
\]
by $\tau:=\sigma\circ\varphi^{-1}$.
%%%%%%%%%%%%%%%%%%%%%%%%%%%%%%%%%%%%%%%%%%%%

\begin{prop}
\label{tau_pol} For $h \in\widetilde{H}^{m}(\mathbf{k})$ and $g \in
\mathbb{C}[P]$, we have
\[
\tau(h)(g) = x^{\rho- \rho^{m}}\pi(h) x^{\rho^{m}-\rho}g.
\]
In particular, the restriction of $\tau$ to $\widetilde{H}^{m}(\mathbf{k})$
preserves $\mathbb{C}[P]$, and the restriction of $\tau$ to $\widetilde{H}%
^{m}(\mathbf{k}, \Lambda_{0})$ preserves $\mathbb{C}[\Lambda]$.
\end{prop}

\begin{proof}
The formula follows from \eqref{sigma_piloc}, Proposition \ref{locrep}%
\textbf{(ii)} and Remark \ref{backtorho}, and then the statements about
restrictions follow from Theorem \ref{mainTHM} and Lemma \ref{pi_conj_stab}.
\end{proof}

%%%%%%%%%%%%%%%%%%%%%%%%%%%%%%%%%%%%%%%%%%%%

\begin{prop}
\label{tauPropaction} We have
\[%
\begin{split}
\tau(T_{i})(fx^{\lambda}) & =\mathbf{k}_{i}fx^{\lambda}+ \mathbf{k}_{i}%
^{-1}c_{i}\bigl(\sigma(s_{i})(fx^{\lambda})-fx^{\lambda}\bigr),\\
\tau(g)(fx^{\lambda}) & =gfx^{\lambda}%
\end{split}
\]
for $f,g\in\mathbb{C}(P^{m})$, $\lambda\in P$ and $i=1,\ldots,r$.
\end{prop}

%%%%%%%%%%%%%%%%%%%%%%%%%%%%%%%%%%%%%%%%%%%%

\begin{proof}
This is immediate from the fact that $\varphi^{-1}(T_{i})=\mathbf{k}%
_{i}+\mathbf{k}_{i}^{-1}c_{i}(s_{i}-1)$ and $\varphi^{-1}(g)=g$ for
$i=1,\ldots,r$ and $g\in\mathbb{C}(P^{m})$.
\end{proof}

%%%%%%%%%%%%%%%%%%%%%%%%%%%%%%%%%%%%%%%%%%%

Define the linear operator
\begin{equation}
\label{CGP_DL}\mathcal{T}_{i}:=-\mathbf{k}_{i}\tau\bigl(x^{\rho^{m}}T_{i}%
^{-1}x^{-\rho^{m}}\bigr) \in\textup{End}(\mathbb{C}(P)).
\end{equation}

%%%%%%%%%%%%%%%%%%%%%%%%%%%%%%%%%%%%%%%%%%%%%

\begin{defi}
We call $\mathcal{T}_{i}\in\textup{End}(\mathbb{C}(P))$ ($i=1,\ldots,r$) the
metaplectic De\-ma\-zure-Lusztig operators.
\end{defi}

%%%%%%%%%%%%%%%%%%%%%%%%%%%%%%%%%%%%%%%%%%
By a direct computation,
\begin{equation}
\label{altT}\mathcal{T}_{i}(f)=(1-\mathbf{k}_{i}^{2}x^{\alpha_{i}^{m}})\left(
\frac{f-x^{\alpha_{i}^{m}}\sigma(s_{i})f} {1-x^{\alpha_{i}^{m}}}\right)
-f,\qquad f\in\mathbb{C}(P).
\end{equation}
They restrict to well-defined linear operators on $\mathbb{C}(\Lambda)$ for
any lattice $\Lambda$ in $V$ satisfying $Q\subseteq\Lambda\subseteq P$, in
which case they reduce for the equal Hecke and representation parameter case
to the Demazure-Lusztig operators \cite[(11)]{CGP}.

%%%%%%%%%%%%%%%%%

\begin{lem}
\label{polpres} The metaplectic Demazure-Lusztig operator $\mathcal{T}_{i}$
stabilizes $\mathbb{C}[P]$ and $\mathbb{C}[\Lambda]$ for $i=1,\ldots,r$.
\end{lem}

%%%%%%%%%%%%%%%%%

\begin{proof}
Follows from \eqref{CGP_DL}, Proposition \ref{tau_pol}, and Lemma
\ref{pi_conj_stab}.
\end{proof}

%%%%%%%%%%%%%%%%%%
The realization \eqref{CGP_DL} of the $\mathcal{T}_{i}$'s through the
$\widetilde{H}^{m}_{\textup{loc}}(\mathbf{k})$-representation $\tau$ directly
imply that the metaplectic Demazure-Lusztig operators $\mathcal{T}_{i}$
($i=1,\ldots,r$) satisfy the braid relations of $W$ and the quadratic Hecke
relations
\begin{equation}
\label{HeckecalT}\mathcal{T}_{i}^{2}=(\mathbf{k}_{i}^{2}-1)\mathcal{T}%
_{i}+\mathbf{k}_{i}^{2},\qquad i=1,\ldots,r
\end{equation}
(this in particular provides an alternative and uniform proof of the braid
relations and quadratic Hecke relations of the metaplectic Demazure-Lusztig
operators in \cite{CGP}, see \cite[Prop. 5(ii)]{CGP} and formula (13) in
\cite[Prop. 7]{CGP}). For $w=s_{i_{1}}\cdots s_{i_{r}}\in W$ a reduced
expression we write $\mathcal{T}_{w}:=\mathcal{T}_{i_{1}}\cdots\mathcal{T}%
_{i_{r}}\in\textup{End}(\mathbb{C}(P))$.

%%%%%%%%%%%%%%%%%%%%%%%%%%%%%%%%%%%%%%%%%%%%%

\begin{rema}
Using that $\sigma(s_{i})f=x^{\rho-\rho^{m}}\pi_{\text{loc}}(s_{i}%
)(x^{\rho^{m}-\rho}f)$ we have
\begin{equation}
\label{tau_pi}\tau(x^{\rho^{m}-\rho}T_{i}x^{\rho-\rho^{m}})f=\mathbf{k}%
_{i}f+\mathbf{k}_{i}^{-1}c_{i}(\pi_{\text{loc}}(s_{i})f-f)\\
=\pi_{\text{loc}}(\varphi^{-1}(T_{i}))f
\end{equation}
for $f\in\mathbb{C}(P)$. Hence
\begin{equation}
\label{altcalT}\mathcal{T}_{i}=-\mathbf{k}_{i}\pi_{\text{loc}}(\varphi
^{-1}(x^{\rho}T_{i}^{-1}x^{-\rho})).
\end{equation}

\end{rema}

%%%%%%%%%%%%%%%%%%%%%%%%%%%%%%%%%%%%%%%%%%%

\begin{rema}
Let $\Lambda$ be a lattice in $E$ satisfying $Q\subseteq\Lambda\subseteq P$.
The localization isomorphism $\varphi$ restricts to an isomorphism of
algebras
\[
\varphi_{\Lambda}: W\ltimes\mathbb{C}(\Lambda_{0})\overset{\sim
}{\longrightarrow} \widetilde{H}^{m}_{\text{loc}}(\mathbf{k},\Lambda_{0}),
\]
with $\widetilde{H}^{m}_{\text{loc}}(\mathbf{k},\Lambda_{0})$ the subalgebra
of $\widetilde{H}^{m}_{\text{loc}}(\mathbf{k})$ generated by $H(\mathbf{k})$
and the quotient field $\mathbb{C}(\Lambda_{0})$ of $\mathbb{C}[\Lambda_{0}]$.
The algebra map
\[
\tau_{\Lambda}: \widetilde{H}^{m}_{\text{loc}}(\mathbf{k},\Lambda
_{0})\rightarrow\textup{End}(\mathbb{C}(\Lambda))
\]
defined by $\tau_{\Lambda}:=\sigma_{\Lambda}\circ\varphi_{\Lambda}^{-1}$ then
satisfies
%%%%%%%%%%%%%%%%%%%%%%%%%%%%%%%%%%%%%%%%%%%%
\[%
\begin{split}
\tau_{\Lambda}(T_{i})(fx^{\lambda}) & =\mathbf{k}_{i}fx^{\lambda}+
\mathbf{k}_{i}^{-1}c_{i}\bigl(\sigma_{\Lambda}(s_{i})(fx^{\lambda
})-fx^{\lambda}\bigr),\\
\tau_{\Lambda}(g)(fx^{\lambda}) & =gfx^{\lambda}%
\end{split}
\]
for $f,g\in\mathbb{C}(\Lambda_{0})$, $\lambda\in\Lambda$ and $i=1,\ldots,r$,
where $\Lambda_{0}:=\Lambda\cap P^{m}$. Note that
\[
\tau_{\Lambda}(X)=\tau(X)|_{\mathbb{C}(\Lambda)},\qquad X\in\widetilde{H}%
^{m}_{\text{loc}}(\mathbf{k},\Lambda_{0}).
\]
The metaplectic Demazure-Lusztig operators $\mathcal{T}_{i}$ then restrict to
the following linear operators on $\mathbb{C}(\Lambda)$,
\[
\mathcal{T}_{i}|_{\mathbb{C}(\Lambda)}=-\mathbf{k}_{i}\tau_{\Lambda
}(\textup{Ad}_{x^{\rho^{m}}}(T_{i}^{-1})),
\]
where $\textup{Ad}_{x^{\rho^{m}}}\in\textup{Aut}(\widetilde{H}^{m}%
_{\text{loc}}(\mathbf{k},\Lambda_{0}))$ is the restriction of the inner
automorphism $X\mapsto x^{\rho^{m}}Xx^{-\rho^{m}}$ of $\widetilde{H}%
^{m}_{\text{loc}}(\mathbf{k}) $ to the subalgebra $\widetilde{H}%
^{m}_{\text{loc}}(\mathbf{k},\Lambda_{0})$.
\end{rema}

%%%%%%%%%%%%%%%%%%%%%%%%

We now use these results to generalize results from \cite[§3]{CGP} to the case
of unequal Hecke and representation parameters. We first analyze certain
symmetrizer and antisymmetrizer elements in $\widetilde{H}^{m}_{\text{loc}%
}(\mathbf{k})$. We then use the metaplectic Weyl group representation $\sigma$
to obtain generalizations of the formula \cite[Thm. 16]{CGP} for the
metaplectic Whittaker function.

Recall from Section \ref{ssec:eaha} that $\mathbf{k} : \Phi\rightarrow
\mathbb{C}^{\times}$ is a $W$-invariant function and $\mathbf{k}_{j} :=
\mathbf{k}_{a_{j}^{m}}$ for $j = 0, \dots, r$. For $w = s_{i_{1}} \cdots
s_{i_{m}}\in W$ a reduced word ($1\leq i_{j}\leq r$), we define
\begin{equation}
\label{k_gen}\mathbf{k}_{w} := \prod_{j=1}^{m} \mathbf{k}_{i_{j}}.
\end{equation}
Note that, in the special case that $\mathbf{k}$ is a constant function (the
equal Hecke algebra parameters case), we have $\mathbf{k}_{w} = \mathbf{k}%
^{\ell(w)}$. Also let
\begin{equation}
\label{k_norm}W(\mathbf{k}^{\pm2}) := \sum_{w \in W} \mathbf{k}_{w}^{\pm2}.
\end{equation}

Define the symmetrizer $\mathbf{1}_{+}\in H(\mathbf{k})$ and antisymmetrizer
$\mathbf{1}_{-}\in H(\mathbf{k})$ by
\begin{equation}
\label{symm}\mathbf{1}_{+} := \sum_{w \in W} \mathbf{k}_{w} T_{w},
\qquad\mathbf{1}_{-} := \sum_{w \in W} (-1)^{l(w)} \mathbf{k}_{w}^{-1} T_{w}.
\end{equation}
It is well known (see e.g., \cite[1.19.1]{Ka} and \cite{Ch}) that the
symmetrizer $\mathbf{1}_{+}$ and antisymmetrizer $\mathbf{1}_{-}$ satisfy the
following properties.
%%%%%%%%%%%%%%%%%%%%%%%%%%%%%%%%%%

\begin{prop}
We have the following identities in $H(\mathbf{k})$:
\begin{equation}
\label{identidem}%
\begin{split}
T_{i} \mathbf{1}_{\pm}  & = \pm\mathbf{k}_{i}^{\pm1} \mathbf{1}_{\pm} =
\mathbf{1}_{\pm} T_{i},\\
\mathbf{1}_{\pm}^{2}  & = W(\mathbf{k}^{\pm2})\mathbf{1}_{\pm}%
\end{split}
\end{equation}
for $i=1,\ldots,r$.
\end{prop}

%%%%%%%%%%%%%%%%%%%%%%%%%%%%%%%
The equations $T_{i}\mathbf{1}_{\pm}=\pm\mathbf{k}_{i}^{\pm1}\mathbf{1}_{\pm}$
for $i=1,\ldots,r$ characterize $\mathbf{1}_{\pm}$ as an element in
$H(\mathbf{k})$ up to a multiplicative constant. It follows from this
observation that
\begin{equation}
\label{symmalt}\mathbf{1}_{+} = \mathbf{k}_{w_{0}}^{2}\sum_{w\in W}%
\mathbf{k}_{w}^{-1}T_{w^{-1}}^{-1}, \qquad\mathbf{1}_{-} = \mathbf{k}_{w_{0}%
}^{-2}\sum_{w\in W}(-1)^{\ell(w)}\mathbf{k}_{w}T_{w^{-1}}^{-1}.
\end{equation}
The multiplicative constant is determined by comparing the coefficient of
$T_{w_{0}}$ in the linear expansion in terms of the basis $\{T_{w}\}_{w\in W}$
of $H(\mathbf{k})$.

Recall the definition \eqref{cfunction} of the $c$-functions $c_{\alpha}$
($\alpha\in\Phi$).

%%%%%%%%%%%%%%%%%%%%%%%%%%%%%%%%%%%%%%%%%%%%

\begin{prop}
\label{fact} We have the following identities in $W\ltimes\mathbb{C}(P^{m})$:
\begin{align*}
\varphi^{-1}(\mathbf{1}_{+})  & = \Bigl(\sum_{w\in W} w\Bigr) \prod_{\alpha
\in\Phi^{+}}c_{-\alpha},\\
\varphi^{-1}(\mathbf{1}_{-})  & = \mathbf{k}_{w_{0}}^{-2}\Bigl(\prod
_{\alpha\in\Phi^{+}}c_{\alpha}\Bigr)\sum_{w\in W} (-1)^{\ell(w)}w.
\end{align*}

\end{prop}

\begin{proof}
See \cite[(5.5.14)]{Ma}.
\end{proof}

We now obtain the following main result of this section.

\begin{thm}
\label{WhittakerDL} We have the following identity of operators in
$\mathrm{End}(\mathbb{C}(P))$:
\[%
\begin{split}
\sum_{w \in W} \mathcal{T}_{w}  & = \Bigl(\prod_{\alpha\in\Phi^{+}}c_{\alpha
}\Bigr)x^{\rho^{m}}\Bigl( \sum_{w\in W}(-1)^{\ell(w)}\sigma(w)\Bigr)x^{-\rho
^{m}}\\
& = \Bigl(\prod_{\alpha\in\Phi^{+}}c_{\alpha}\Bigr)\sum_{w\in W}
(-1)^{\ell(w)}\Bigl(\prod_{\alpha\in\Phi(w^{-1})}x^{\alpha^{m}}\Bigr)\sigma
(w).
\end{split}
\]
In particular, for $\lambda\in P^{+}$ we have
\[
\widetilde{\mathcal{W}}_{\lambda}=\sum_{w\in W}\mathcal{T}_{w}(x^{w_{0}%
\lambda}).
\]

\end{thm}

\begin{proof}
By \eqref{CGP_DL} and \eqref{symmalt} we have
\begin{align*}
\sum_{w \in W} \mathcal{T}_{w}  & = \sum_{w\in W}(-1)^{\ell(w)}\mathbf{k}_{w}
\tau\bigl(x^{\rho^{m}}T_{w^{-1}}^{-1}x^{-\rho^{m}}\bigr)\\
& = \mathbf{k}_{w_{0}}^{2}\tau\bigl(x^{\rho^{m}}\mathbf{1}_{-}x^{-\rho^{m}%
}\bigr).
\end{align*}
The first formula now follows directly using $\tau=\sigma\circ\varphi^{-1}$
and the previous proposition. The second formula follows from the observation
that
\[
\rho^{m}-w\rho^{m}=\sum_{\alpha\in\Phi(w^{-1})}\alpha^{m}%
\]
for $w\in W$.
\end{proof}

%%%%%%%%%%%%%%%%%%%%%%%%%%%%%%%%%%%%%%%%%%%%%

\begin{cor}
Let $\Lambda\subset E$ be a lattice satisfying $Q\subseteq\Lambda\subseteq P$.
Then $\widetilde{\mathcal{W}}_{\lambda}\in\mathbb{C}[\Lambda]$ for $\lambda
\in\Lambda^{+}:=P^{+}\cap\Lambda$.
\end{cor}

%%%%%%%%%%%%%%%%%%%%%%%%%%%%%%%%%%%%%%%%%%%%%%

\begin{proof}
This follows from Lemma \ref{polpres} and the previous theorem.
\end{proof}

%%%%%%%%%%%%%%%%%%%%%%%%%%%%%%%%%%%%%%%%%%%%%%

\begin{rema}
\label{metHL} Note that the symmetric variant $\tau(\mathbf{1}_{+}%
)(x^{\lambda})$ of $\widetilde{\mathcal{W}}_{\lambda}$ for $\lambda\in
\Lambda^{+}$ may also be of interest. These are polynomials (again by Lemma
\ref{polpres}), symmetric with respect to the Chinta-Gunnells $W$-action
$\sigma$ (by Proposition \ref{fact}(a)), which reduce for $m\equiv1$ to
Hall-Littlewood polynomials \cite[§10]{Mac} (by e.g., Remark \ref{standard_W}).
\end{rema}

%%%%%%%%%%%%%%%%%%%%%%%%%%%%%%%%%%%%%%%%%%%%%%

\begin{rema}
In the equal Hecke and parameter case, $\widetilde{\mathcal{W}}_{\lambda}$
admits the interpretation as a metaplectic Whittaker function attached to a
metaplectic cover of a reductive group over a nonarchimedean local field, see
\cite[Thm. 16]{CGP}. It is a natural open problem what the corresponding
representation theoretic interpretation is of $\widetilde{\mathcal{W}%
}_{\lambda}$ in the unequal Hecke and/or representation parameter case.
\end{rema}

%%%%%%%%%%%%%%%%%%%%%%%%%%%%%%%%%%%%%%%%%%%%%%%

\section{Metaplectic polynomials}

\label{sect5}

In this section we present metaplectic variants of $\textup{GL}_{r}$ Macdonald
polynomials. Full proofs and additional results will be provided in a
forthcoming paper, in which we will also introduce the metaplectic polynomials
for arbitrary root systems.

%%%%%%%%%%%%%%%%%%%%%%%%%%%

\subsection{The metaplectic data $(n,\mathbf{Q})$}

%%%%%%%%%%%%%%%%%%%%%%%%%%%
Let $r\geq2$. Fix the standard orthonormal basis $\{e_{i}\}_{i=1}^{r}$ of
$\mathbb{R}^{r}$. The associated scalar product is denoted by $\bigl(\cdot
,\cdot\bigr)$ and the corresponding norm by $\|\cdot\|$. Then
\[
\Phi=\{\epsilon_{i}-\epsilon_{j}\}_{1\leq i\not =j\leq r}%
\]
is the root system of type $\textup{A}_{r-1}$, with basis $\Delta$ of simple
roots and associated set $\Phi^{+}$ of positive roots given by
\[
\Delta:=\{\alpha_{1},\ldots,\alpha_{r-1}\}\subset\Phi^{+}=\{\epsilon
_{i}-\epsilon_{j}\}_{1\leq i<j\leq r}%
\]
with $\alpha_{i}:=\epsilon_{i}-\epsilon_{i+1}$. The associated highest root is
$\theta=\epsilon_{1}-\epsilon_{r}$. The root lattice is $Q:=\mathbb{Z}\Phi$,
which is contained in the $\textup{GL}_{r}$ weight lattice $\bigoplus
_{i=1}^{r}\mathbb{Z}\epsilon_{i}\simeq\mathbb{Z}^{r}$. The Weyl group is the
symmetric group $S_{r}$ in $r$ letters.

Let $\mathbf{Q}: \mathbb{Z}^{r}\rightarrow\mathbb{Q}$ be a non-zero $S_{r}%
$-invariant quadratic form which is integral-valued on $Q$. Then
\[
\mathbf{Q}(\gamma)=\frac{\kappa}{2}\|\gamma\|^{2}\qquad\forall\,\gamma\in Q
\]
for some nonzero integer $\kappa=\kappa_{\mathbf{Q}}$ (we suppress the
dependence of $\kappa$ on $\mathbf{Q}$ if it is clear from the context). In
particular, $\mathbf{Q}(\alpha)=\kappa$ for all $\alpha\in\Phi$. Write
\[
\mathbf{B}(\lambda,\mu):=\mathbf{Q}(\lambda+\mu)-\mathbf{Q}(\lambda
)-\mathbf{Q}(\mu),\qquad\lambda,\mu\in\mathbb{Z}^{r}%
\]
for the associated symmetric $S_{r}$-invariant bilinear form $\mathbf{B}:
\mathbb{Z}^{r}\times\mathbb{Z}^{r}\rightarrow\mathbb{Q}$. By the $S_{r}%
$-invariance of $\mathbf{B}$ we then have
\begin{equation}
\label{Bexpl}\mathbf{B}(\lambda,\alpha)=\kappa\bigl(\lambda,\alpha^{\vee
}\bigr)\qquad\forall\, \lambda\in\mathbb{Z}^{r},\,\,\, \forall\,\alpha\in\Phi
\end{equation}
(in the present context $\alpha^{\vee}=\alpha$ for $\alpha\in\Phi$, but we
distinguish them in anticipation of the results for arbitrary root systems in
our followup paper). In particular, $\mathbf{B}(\lambda,\alpha)\in\mathbb{Z}$
for $\lambda\in\mathbb{Z}^{r}$ and $\alpha\in\Phi$.

Fix $n\in\mathbb{Z}_{>0}$ once and for all. Given a quadratic form
$\mathbf{Q}$ as in the previous paragraph with associated normalisation scalar
$\kappa=\kappa_{\mathbf{Q}}$, we define positive integers $\kappa^{\prime
}=\kappa^{\prime}_{\mathbf{Q}}$ and $m=m_{\mathbf{Q}}$ by
\begin{equation}
\label{defm}\kappa^{\prime}:=\textup{gcd}(n,\kappa),\qquad m:=\frac{n}%
{\kappa^{\prime}}=\frac{n}{\textup{gcd}(n,\kappa)}.
\end{equation}
Note that $m=n/\textup{gcd}(n,\mathbf{Q}(\alpha))$ for all $\alpha\in\Phi$, in
particular $\Phi^{m}=m\Phi$. Furthermore,
\[
m\mathbb{Z}^{r}\subseteq\{\lambda\in\mathbb{Z}^{r} \,\, | \,\, \mathbf{B}%
(\lambda,\alpha)\equiv0 \quad\textup{mod } n\,\,\, \forall\, \alpha\in\Phi\}.
\]

Set $\mathbb{F}:=\mathbb{C}(q,k)$. We consider the following field extensions
$\mathbb{F}\subseteq\mathbb{K}^{(n)}$,
\[
\mathbb{K}^{(n)}:=
\begin{cases}
\mathbb{F}(g_{1}^{(n)},\ldots,g_{\lfloor\frac{n}{2}\rfloor}^{(n)}),\qquad &
\hbox{ if } n \hbox{ is odd},\\
\mathbb{F}(g_{1}^{(n)},\ldots,g_{\frac{n}{2}-1}^{(n)}),\qquad & \hbox{ if } n
\hbox{ is even},
\end{cases}
\]
which should be read as $\mathbb{K}^{(n)}=\mathbb{F}$ for $n=1,2$.

For $n\geq1$ we now define representation parameters $g_{j}^{(n)}\in
\mathbb{K}^{(n)}$ for all integers $j\in\mathbb{Z}$ as follows (it depends on
a choice of a sign $\epsilon\in\{\pm1\}$ when $n$ is even, which we fix once
and for all). We set $g_{0}^{(n)}:=-1$. The representation parameters
$g_{j}^{(n)}$ for indices $\frac{n}{2}<j<n$ are defined by $g_{j}%
^{(n)}:=k^{-2}(g_{n-j}^{(n)})^{-1}$. For $n$ even, we set $g_{\frac{n}{2}%
}^{(n)}:=\epsilon^{-1}k^{-1}$. Finally, the representation parameters
$g_{j}^{(n)}\in\mathbb{K}^{(n)}$ are extended to indices $j\in\mathbb{Z}$ by
$g_{j}^{(n)}:=g_{r_{n}(j)}^{(n)}$, with $r_{n}(j)\in\{0,\ldots,n-1\}$ the
remainder modulo $n$. Note that $g_{j}^{(n)}g_{n-j}^{(n)}=k^{-2}$ in
$\mathbb{K}^{(n)}$ for all $j\in\mathbb{Z}\setminus n\mathbb{Z}$, and
$g_{j}^{(n)}=-1$ if $j\in n\mathbb{Z}$.
%%%%%%%%%%%%%%%%%%%%%%%%%%%%%%%%%%%%%%%

\begin{lem}
\label{iota} There exists a unique $\mathbb{F}$-homomorphism $\iota_{\kappa}:
\mathbb{K}^{(m)}\hookrightarrow\mathbb{K}^{(n)}$ mapping $g_{j}^{(m)}$ to
$g_{\kappa j}^{(n)}$ for all $j\in\mathbb{Z}$.
\end{lem}

%%%%%%%%%%%%%%%%%%%%%%%%%%%%%%%%%%%%%%%%

\begin{proof}
This is an easy check.
\end{proof}

%%%%%%%%%%%%%%%%%%%%%%%%%%%%%%%%%%%%%%%%%
We write $\mathbb{K}^{(n,\kappa)}$ for the image of $\mathbb{K}^{(m)}=
\mathbb{K}^{(n/\kappa^{\prime})}$ under $\iota_{\kappa}: \mathbb{K}%
^{(m)}\hookrightarrow\mathbb{K}^{(n)}$. It is the subfield of $\mathbb{K}%
^{(n)}$ obtained by adjoining the elements $g_{\kappa^{\prime}j}^{(n)}$ to
$\mathbb{F}$ for $1\leq j<\frac{m}{2}$. Note that $\mathbb{K}^{(n,1)}%
=\mathbb{K}^{(n)}$.

We finish this subsection by introducing metaplectic analogues $(p_{j})_{0\leq
j<r}$ of multiplicity functions. For $\lambda\in\mathbb{Z}^{r}$ write
$\overline{\lambda}^{m}=\lambda+m\mathbb{Z}^{r}$ for the class of $\lambda$ in
$\bigl(\mathbb{Z}/m\mathbb{Z}\bigr)^{r}$. Then we define
\[
p_{j}: \bigl(\mathbb{Z}/m\mathbb{Z}\bigr)^{r}\rightarrow\mathbb{K}^{(n)}%
\qquad(0\leq j<r)
\]
by
\[
p_{i}(\overline{\lambda}^{m}):=-kg_{-\mathbf{B}(\lambda,\alpha_{i})}^{(n)},
\qquad p_{0}(\overline{\lambda}^{m}):=-kg_{\mathbf{B}(\lambda,\theta)}^{(n)}%
\]
for $1\leq i<r$. By \eqref{Bexpl}, the dependence of $p_{j}$ on the
metaplectic data is a dependence on $(n,\kappa)$ (and $\epsilon$ if $n$ is
even). If we want to emphasize it, we will write $p_{j}=p_{j}^{(n,\kappa)}$
(we always suppress $\epsilon$ from the notations). Note that by
\eqref{Bexpl}, the functions $p_{j}$ take values in the subfield
$\mathbb{K}^{(n,\kappa)}$ of $\mathbb{K}^{(n)}$.

%%%%%%%%%%%%%%%%%%%%%%%%%%%%

\subsection{The double affine Hecke algebra $\mathbb{H}^{(m)}$}\label{sect:DAHA}

%%%%%%%%%%%%%%%%%%%%%%%%%%%%
Consider the extended affine Weyl group $W^{(m)}:=S_{r}\ltimes m\mathbb{Z}%
^{r}$. We denote its elements by $\sigma\tau(\nu)$ ($\sigma\in S_{r}$, $\nu\in
m\mathbb{Z}^{r}$). We may also view $W^{(m)}$ as the subgroup of affine linear
transformations of $\mathbb{R}^{r}$ of the form $\sigma\tau(\nu)$ ($\sigma\in
S_{r}$, $\nu\in m\mathbb{Z}^{r}$), acting on $\mathbb{R}^{r}$ by
\[
(\sigma\tau(\nu))(v):=\sigma(v+\nu),\qquad v\in\mathbb{R}^{r},\,\, \sigma\in
S_{r},\,\, \nu\in m\mathbb{Z}^{r}.
\]
View $\mathbb{R}^{r}\oplus\mathbb{R}$ as the space of real-valued affine
linear functionals on $\mathbb{R}^{r}$ by associating to $(v,x)\in
\mathbb{R}^{r}\oplus\mathbb{R}$ the affine linear functional $\mathbb{R}%
^{r}\ni u\mapsto(v,u)+x$. The extended affine Weyl group $W^{(m)}$ acts on
$\mathbb{R}^{r}\oplus\mathbb{R}$ by
\[
(\sigma\tau(\nu))(v,x):=(\sigma v, x-(\nu,v)).
\]

The affine root system is
\[
\widetilde{\Phi}^{(m)}=\{(m\alpha,tm^{2})\,\,\, | \,\,\, \alpha\in\Phi,\,\,\,
t\in\mathbb{Z} \} \subset\mathbb{R}^{r}\oplus\mathbb{R},
\]
which is stabilized by $W^{(m)}$. We identify $m\Phi$ with the subset of
affine linear roots $\{(m\alpha,0)\}_{\alpha\in\Phi}$ in $\widetilde{\Phi
}^{(m)}$. For $a=(m\alpha,tm^{2})\in\widetilde{\Phi}^{(m)}$ let $s_{a}\in
W^{(m)}$ be the orthogonal reflection in the affine hyperplane $a^{-1}(0)$.
Then
\[
s_{a}=\tau(-tm\alpha^{\vee})s_{\alpha}\in W^{(m)}%
\]
with $s_{\alpha}\in S_{r}$ the orthogonal reflection in the hyperplane
$\alpha^{\perp}\subset\mathbb{R}^{r}$.

We take
\[
\{b_{0}^{(m)},b_{1}^{(m)},\ldots,b_{r-1}^{(m)}\}:=\{(-m\theta,m^{2}%
),m\alpha_{1},\ldots,m\alpha_{r-1}\}
\]
as the set of simple roots of $\widetilde{\Phi}^{(m)}$ and write $s_{j}%
^{(m)}:=s_{b_{j}}\in W^{(m)}$ ($j=0,\ldots,r-1$) for the associated simple
reflections. Then
\[
s_{0}^{(m)}=\tau(m\theta^{\vee})s_{\theta},
\]
and $s_{i}^{(m)}=s_{\alpha_{i}}\in S_{r}$ ($1\leq i<r$) are the simple
neighbouring transpositions. Since the latter do not depend on $m$, we will
write $s_{i}=s_{i}^{(m)}$ for $1\leq i<r$.

The subgroup $W^{(m)}_{\textup{Cox}}:=\langle s_{0}^{(m)},\ldots,s_{r-1}%
^{(m)}\rangle$ of $W^{(m)}$ is the affine Weyl group of type
$\widehat{\textup{A}}_{r-1}$. Its defining relations in terms of the simple
reflections are $(s_{j}^{(m)})^{2}=1$ and the type $\widehat{A}_{r-1}$ braid
relations. Then $W^{(m)}\simeq\mathbb{Z}\ltimes W^{(m)}_{\textup{Cox}} $, with
$1\in\mathbb{Z}$ acting on $W^{(m)}_{\textup{Cox}}$ by $s_{j}^{(m)}\mapsto
s_{j+1}^{(m)}$ (indices modulo $r$), which corresponds under the isomorphism
$W^{(m)}\simeq\mathbb{Z}\ltimes W^{(m)}_{\textup{Cox}} $ with the extended
affine Weyl group element
\[
\omega^{(m)}:=s_{1}s_{2}\cdots s_{r-1}\tau(m\epsilon_{r}).
\]
Note that $\omega^{(m)}(b_{j}^{(m)})=b_{j+1}^{(m)}$ for $0\leq j<r$ (with the
indices taken modulo $r$).

We write $x^{v}$ ($v\in\mathbb{R}^{r}$) for the canonical basis of the group
algebra $\mathbb{F}[\mathbb{R}^{r}]$ of $\mathbb{R}^{r}$ over $\mathbb{F}$, so
that $x^{u}x^{v}=x^{u+v}$ and $x^{0}=1$. We write for $c\in\mathbb{Z}$ and
$v\in\mathbb{R}^{r}$,
\[
x^{(v,c)}=q^{c}x^{v}\in\mathbb{F}[\mathbb{R}^{r}].
\]
Let $\mathbb{F}[x^{\pm1}]$ be the $\mathbb{F}$-algebra of Laurent polynomials
in $x_{1},\ldots,x_{r}$, viewed as the $\mathbb{F}$-subalgebra of
$\mathbb{F}[\mathbb{R}^{r}]$ generated by $\mathbb{Z}^{r}\subset\mathbb{R}%
^{r}$ via $x_{i}:=x^{\epsilon_{i}}$ ($1\leq i\leq r$). The extended affine
Weyl group $W^{(m)}$ acts by $\mathbb{F}$-algebra automorphisms on
$\mathbb{F}[x^{\pm1}] $ by
\begin{equation}
\label{v1}w\bigl(x^{(\lambda,c)}\bigr):=x^{w(\lambda,c)}%
\end{equation}
for $w\in W^{(m)}$ and $(\lambda,c)\in\mathbb{Z}^{r}\oplus\mathbb{Z}$. In
particular, for $\lambda\in\mathbb{Z}^{r}$, $\sigma\in S_{r}$ and $\nu
\in\mathbb{Z}^{r}$,
\begin{equation}
\label{v2}(\sigma\tau(\nu))x^{\lambda}=q^{-(\nu,\lambda)}x^{\sigma\lambda}.
\end{equation}

For $\lambda\in\mathbb{Z}^{r}$ we thus have
\[
x^{\omega^{(m)}\lambda}=q^{-m\lambda_{r}}x^{s_{1}\cdots s_{r-1}\lambda},
\qquad x^{s_{0}^{(m)}\lambda}=q^{m(\lambda,\theta^{\vee})}x^{s_{\theta}%
\lambda}%
\]
and $x^{b_{0}^{(m)}}=q^{m^{2}}x^{-m\theta}$.
%%%%%%%%%%%%%%%%%%%%%%%%%%%%%%%%%%%%%%%%%%

\begin{defi}\label{def:DAHA}
The $\textup{GL}_{r}$ double affine Hecke algebra $\mathbb{H}^{(m)}$ is the
unital associative $\mathbb{F}$-algebra generated by $T_{0},\ldots,T_{r-1}$,
$\omega^{\pm1}$ and $\mathbb{F}[x^{\pm m}]:=\mathbb{F}[x_{1}^{\pm m}%
,\ldots,x_{r}^{\pm m}]$ with defining relations:

\begin{enumerate}
\item The type $\widehat{\textup{A}}_{r-1}$ braid relations for $T_{1}%
,\ldots,T_{r-1}$.

\item The Hecke relations $(T_{j}-k)(T_{j}+k^{-1})=0$.

\item $\omega\omega^{-1}=1=\omega^{-1}\omega$ and $\omega T_{j}=T_{j+1}\omega$
(indices modulo $r$).

\item The cross relations
\begin{equation}
\label{crossdaha}T_{j}x^{\lambda}-x^{s^{(m)}_{j}\lambda}\,T_{j}=(k-k^{-1}%
)\left( \frac{x^{\lambda}-x^{s_{j}^{(m)}\lambda}} {1-x^{b_{j}^{(m)}}}\right)
,\qquad\omega x^{\lambda}=x^{\omega^{(m)}\lambda}\omega
\end{equation}
for $\lambda\in m\mathbb{Z}^{r}$ and $0\leq j<r$.
\end{enumerate}
\end{defi}

%%%%%%%%%%%%%%%%%%%%%%%%%%%%%%%%%%%%%%%%
Consider the subalgebras $\widetilde{H}^{(m)}_{Y}:=\mathbb{F}\langle
T_{0},\ldots,T_{r-1},\omega^{\pm1}\rangle$ and $H^{(m)}=\mathbb{F}\langle
T_{1},\ldots,T_{r-1}\rangle$ of $\mathbb{H}^{(m)}$. The subalgebra $H^{(m)}$
is the finite Hecke algebra (of type $A_{r-1}$). Define
\begin{equation}
\label{eqref:Y-opers}Y^{m\epsilon_{i}}:=T_{i-1}^{-1}\cdots T_{1}^{-1}\omega
T_{r-1}\cdots T_{i}, \qquad i=1,\ldots,r.
\end{equation}
They pairwise commute and are invertible in $\widetilde{H}^{(m)}_{Y}$. The
assignment
\[
x^{\nu}\mapsto Y^{\nu}:=Y_{1}^{\nu_{1}}\cdots Y^{\nu_{r}}_{r},\qquad\nu\in
m\mathbb{Z}^{r}%
\]
defines an injective algebra map $\mathbb{F}[x^{\pm m}]\hookrightarrow
H^{(m)}$, whose image we denote by $\mathbb{F}[Y^{\pm m}]$. The multiplication
map
\[
H^{(m)}\otimes_{\mathbb{F}}\mathbb{F}[Y^{\pm m}]\rightarrow\widetilde{H}%
^{(m)}_{Y},\qquad h\otimes Y^{\lambda}\mapsto hY^{\lambda}%
\]
s a $\mathbb{F}$-linear isomorphism. The defining relations of $\widetilde{H}%
^{(m)}_{Y}$ in terms of the subalgebras $H^{(m)}$ and $\mathbb{F}[Y^{\pm m}]$
are the Bernstein-Zelevinsky cross relations
\[
T_{i}Y^{\mu}-Y^{s_{i}\mu}T_{i}=(k-k^{-1})\left( \frac{Y^{\mu}-Y^{s_{i}\mu}%
}{1-Y^{-m\alpha_{i}}}\right)
\]
for $1\leq i<r$ and $\mu\in m\mathbb{Z}^{r}$.

%%%%%%%%%%%%%%%%%%

\begin{rema}
Let $\delta: \mathbb{H}^{(m)}\rightarrow\mathbb{H}^{(m)}$ be the $\mathbb{F}%
$-linear antialgebra isomorphism satisfying $\delta(T_{i}):=T_{i}$ ($1\leq
i<r$), $\delta(Y^{\mu}):=x^{-\mu}$ and $\delta(x^{\mu}):=Y^{-\mu}$ for $\mu\in
m\mathbb{Z}^{r}$. It provides an anti-isomorphism between the subalgebras
\[
\widetilde{H}_{X}^{(m)}:= \mathbb{F}\langle T_{1},\ldots,T_{r-1},x_{1}^{\pm
m},\ldots,x_{r}^{\pm m}\rangle
\]
and $\widetilde{H}_{Y}^{(m)}$. The corresponding Coxeter type presentation of
$\widetilde{H}^{(m)}_{X}$ thus involves $\delta(T_{0})$ and $\delta
(\omega^{-1})$ as the generator for the simple affine root and the element
corresponding to generator of the affine Dynkin diagram automorphisms. Note
that
\[
\delta(\omega^{-1})=x_{1}^{m}T_{1}\cdots T_{r-1}%
\]
and that $\delta(\omega^{-1})Y^{\lambda}=q^{m\lambda_{r}}Y^{s_{1}\cdots
s_{r-1}\lambda}\delta(\omega^{-1})$ for $\lambda\in m\mathbb{Z}^{r}$ in
$\mathbb{H}^{(m)}$.
\end{rema}

%%%%%%%%%%%%%%%%%%%

%%%%%%%%%%%%%%%%%%%%%%%%%

\subsection{The metaplectic basic representation}\label{sect:metbasic}

%%%%%%%%%%%%%%%%%%%%%%%%%
Set $t_{m}(s):=s-r_{m}(s)\in m\mathbb{Z}$. The metaplectic divided difference
operators $\overline{\nabla}_{j}^{(m)}$ ($0\leq j<r$) are the $\mathbb{F}%
$-linear operators on $\mathbb{F}[x^{\pm1}]$ defined by
\[
\overline{\nabla}_{i}^{(m)}(x^{\lambda}):=\left( \frac{1-x^{-t_{m}%
((\lambda,\alpha_{i}^{\vee}))\alpha_{i}}}{1-x^{m\alpha_{i}}} \right)
x^{\lambda}%
\]
and
\[
\overline{\nabla}_{0}^{(m)}(x^{\lambda}):= \left( \frac{1-q^{-mt_{m}%
(-(\lambda,\theta^{\vee}))}x^{t_{m}(-(\lambda,\theta^{\vee}))\theta}}
{1-q^{m^{2}}x^{-m\theta}}\right) x^{\lambda}%
\]
for $\lambda\in\mathbb{Z}^{r}$. Note that the $\overline{\nabla}_{j}%
^{(m)}|_{\mathbb{F}[x^{\pm m}]}$ are the usual divided-difference operators,
\[
\overline{\nabla}_{j}^{(m)}(x^{\lambda})=\frac{x^{\lambda}-x^{s_{j}%
^{(m)}\lambda}}{1-x^{b_{j}^{(m)}}},\qquad\lambda\in m\mathbb{Z}^{r}%
\]
for $0\leq j<r$. For a field extension $\mathbb{F}\subseteq\mathbb{K}$, the
$\mathbb{K}$-linear extension of $\overline{\nabla}^{(m)}_{j}$ to a linear
operator on $\mathbb{K}[x^{\pm1}]$ will also be denoted by $\overline{\nabla
}_{j}^{(m)}$.

Recall that the metaplectic data $(n,\mathbf{Q})$ provide us with the nonzero
integer $\kappa:=\mathbf{Q}(\alpha)$ ($\alpha\in\Phi$), from which
$\kappa^{\prime}$ and $m$ are determined by \eqref{defm}. In addition, we have
fixed a sign $\epsilon\in\{\pm1\}$ in case $n$ is even through the definition
of the representation parameter $g_{\frac{n}{2}}^{(n)}=\epsilon^{-1}k^{-1}$.

%%%%%%%%%%%%%%%%%%%%%%%%%%%%%%%%%%%

\begin{thm}
\label{GLrmetabasic} The formulas
\begin{equation}
\label{assignments}%
\begin{split}
\widehat{\pi}^{(n,\kappa)}(T_{j})x^{\lambda} & :=(k-k^{-1})\overline{\nabla
}_{j}^{m}(x^{\lambda})+ p_{j}^{(n,\kappa)}(\overline{\lambda}^{m}%
)x^{s_{j}^{(m)}\lambda},\\
\widehat{\pi}^{(n,\kappa)}(x^{\mu})x^{\lambda} & :=x^{\lambda+\mu},\\
\widehat{\pi}^{(n,\kappa)}(\omega)x^{\lambda} & :=x^{\omega^{(m)}\lambda}%
\end{split}
\end{equation}
for $j=0,\ldots,r-1$, $\mu\in m\mathbb{Z}^{r}$ and $\lambda\in\mathbb{Z}^{r}$
turn $\mathbb{K}^{(n,\kappa)}[x^{\pm1}]$ into a left $\mathbb{H}^{(m)}%
$-module. We write $\widehat{\pi}^{(m)}:=\widehat{\pi}^{(m,1)}$ (we suppress
here the dependence on $\epsilon$).
\end{thm}

%%%%%%%%%%%%%%%%%%%%%%%%%%%%%%%%%%%%%%%%%%%

%%%%%%%%%%%%%%%%%%%%%%%%%%%%%%%%%%%%%%%%%%%

\begin{proof}
Set
\[
\Lambda:=\mathbb{Z}^{r}+\mathbb{Z}\mathfrak{n}%
\]
with $\mathfrak{n}:=\frac{1}{r}(\epsilon_{1}+\cdots+\epsilon_{r})$. Note that
$\Lambda$ contains the weight lattice $P$ of $\Phi$. The quadratic form
$\mathbf{Q}$ has a unique extension to a $\mathbb{Q}$-valued $S_{r}$-invariant
quadratic form $\Lambda\rightarrow\mathbb{Q}$, which we also denote by
$\mathbf{Q}$. We write $\mathbf{B}: \Lambda\times\Lambda\rightarrow\mathbb{Q}$
for the associated symmetric $S_{r}$-invariant bilinear form.

Adjoin a $r$th root $q^{\frac{1}{r}}$ of $q$ to $\mathbb{K}^{(n,\kappa)}$ (by
abuse of notation, we denote it again by $\mathbb{K}^{(n,\kappa)}$). Let
$\mathbb{K}^{(n,\kappa)}[\Lambda]$ be the $\mathbb{K}^{(n,\kappa)}$-submodule
of $\mathbb{K}^{(n,\kappa)}[\mathbb{R}^{r}]:=\mathbb{K}^{(n,\kappa)}%
\otimes_{\mathbb{F}}\mathbb{F}[\mathbb{R}^{r}]$ generated by $x^{\lambda}$
($\lambda\in\Lambda$). The extended affine Weyl group $W^{(m)}$ acts on
$\mathbb{K}^{(n,\kappa)}[\Lambda]$ by $\mathbb{K}^{(n,\kappa)}$-algebra
automorphisms by the formula \eqref{v1}. The $\mathbb{K}^{(n,\kappa)}%
$-subalgebra $\mathbb{K}^{(n,\kappa)}[P]$ generated by $x^{\lambda}$
($\lambda\in P$) is a $W^{(m)}$-submodule.

By Theorem \ref{mainTHM} (which holds true with formal parameters), the first
two lines of \eqref{assignments},
\begin{equation}
\label{assignments2}%
\begin{split}
\widehat{\pi}(T_{i})x^{\lambda} & :=(k-k^{-1})\overline{\nabla}_{i}%
^{(m)}(x^{\lambda})+ p_{i}^{(n,\kappa)}(\overline{\lambda}^{m})x^{s_{i}%
\lambda},\\
\widehat{\pi}(x^{\mu})x^{\lambda} & :=x^{\lambda+\mu}%
\end{split}
\end{equation}
for $1\leq i<r$, $\mu\in mP$ and $\lambda\in P$ define a representation
\[
\widehat{\pi}: \widetilde{H}^{(m)}_{X}\rightarrow\textup{End}_{\mathbb{K}%
^{(n,\kappa)}}\bigl(\mathbb{K}^{(n,\kappa)}[P]\bigr).
\]
Using the decomposition
\[
\mathbb{K}^{(n,\kappa)}[\Lambda]= \bigoplus_{s\in\mathbb{Z}}x^{s\mathfrak{n}%
}\mathbb{K}^{(n,\kappa)}[P]
\]
it extends to a representation $\widehat{\pi}: \widetilde{H}^{(m)}%
_{X}\rightarrow\textup{End}_{\mathbb{K}^{(n,\kappa)}}\bigl(\mathbb{K}%
^{(n,\kappa)}[\Lambda]\bigr)$ by $\widehat{\pi}(h)\bigl(x^{s\mathfrak{n}%
}x^{\lambda}\bigr):=x^{s\mathfrak{n}}\widehat{\pi}(h)x^{\lambda}$ for
$h\in\widetilde{H}^{(m)}_{X}$, $s\in\mathbb{Z}$ and $\lambda\in P$. The
formulas \eqref{assignments2} are then valid for all $\lambda\in\Lambda$. A
direct check shows that the operators $\widehat{\pi}(T_{j})$ and
$\widehat{\pi}(x^{\mu}) $ on $\mathbb{K}^{(n,\kappa)}[\Lambda]$ satisfy the
cross relation \eqref{crossdaha} for $1\leq j<r$ and $\mu\in m\mathbb{Z}^{r}$.
Furthermore, direct computations show that
\[%
\begin{split}
p_{j+1}^{(n,\kappa)}\bigl(\overline{(s_{1}\cdots s_{r-1}\lambda)}^{m}\bigr) &
=p_{j}(\overline{\lambda}^{m}),\\
\widehat{\pi}(\omega)\overline{\nabla}_{j}^{(m)} & =\overline{\nabla}%
_{j+1}^{(m)}\widehat{\pi}(\omega)
\end{split}
\]
for $\lambda\in\mathbb{Z}^{r}$ and $0\leq j<r$, hence $\widehat{\pi}%
(\omega)\widehat{\pi}(T_{j})=\widehat{\pi}(T_{j+1})\widehat{\pi}(\omega)$ as
operators on $\mathbb{K}^{(n,\kappa)}[x^{\pm1}]$, with the indices modulo $r$.
{}From this the defining double affine Hecke algebra relations involving
$T_{0}$ are easily verified. The result now follows directly.
\end{proof}

%%%%%%%%%%%%%%%%%%%%%%%%%%%%%%%%%%%%%%%%%%%%
We call $\widehat{\pi}^{(n,\kappa)}$ the \textit{metaplectic basic
representation} of the double affine Hecke algebra $\mathbb{H}^{(m)}$.
%%%%%%%%%%%%%%%%%%%%%%%%%%%%%%%%%%%%%%%%%%%%

\begin{rema}
\label{mnk} \textbf{(i)} Since the double affine Hecke algebra $\mathbb{H}%
^{(m)}$ is defined over $\mathbb{F}$, the representation parameters should be
thought of as representation parameters of the representation $\widehat{\pi
}^{(n,\kappa)}$. Since the representation is defined over the subfield
$\mathbb{K}^{(n,\kappa)}$ of $\mathbb{K}^{(n)}$, the representation
$\widehat{\pi}^{(n,\kappa)}$ only depends on the representation parameters
$g_{\kappa^{\prime}j}^{(n)}$ ($1\leq j<\frac{m}{2}$) and, if $m$ is even, on
$\epsilon$.\newline\textbf{(ii)} It follows from
\[
\iota_{\kappa}(p_{j}^{(m,1)}(\overline{\lambda}^{m}))=p_{j}^{(n,\kappa
)}(\overline{\lambda}^{m})
\]
for $0\leq j<r$ and $\lambda\in\mathbb{Z}^{r}$ that
\[
\overline{\iota}_{\kappa}\bigl(\sum_{\mu}c_{\mu}x^{\mu}\bigr):=\sum_{\mu}%
\iota_{\kappa}(c_{\mu})x^{\mu}\qquad(c_{\mu}\in\mathbb{K}^{(m)})
\]
defines an isomorphism
\[
\overline{\iota}_{\kappa}: (\mathbb{K}^{(m)}[x^{\pm1}], \widehat{\pi}%
^{(m)})\overset{\sim}{\longrightarrow} (\mathbb{K}^{(n,\kappa)}[x^{\pm
1}],\widehat{\pi}^{(n,\kappa)})
\]
of $\mathbb{H}^{(m)}$-modules. In particular, $(\mathbb{K}^{(n,\kappa)}%
[x^{\pm1}],\widehat{\pi}^{(n,\kappa)})$ only depends on $\epsilon$ if $m$ is
even.\newline\textbf{(iii)} $\widehat{\pi}^{(1)}: \mathbb{H}^{(1)}%
\rightarrow\textup{End}_{\mathbb{F}}\bigl(\mathbb{F}[x^{\pm1}]\bigr)$ is
Cherednik's basic representation for $\textup{GL}_{r}$, see, e.g.,
\cite[§3.7]{Ch} and \cite{HHL}.\newline
\end{rema}

%%%%%%%%%%%%%%%%%%%%%%%%%%%%%%%%%%%%%%%%%%%%
By the second part of the remark, the dependence of the metaplectic basic
representation on the metaplectic data is essentially only a dependence on $m
$. The metaplectic basic representation $\widehat{\pi}^{(m)}$ can be recovered
from $\widehat{\pi}^{(n)}$ as follows.

By a direct check one verifies that the assignments
\[
\phi_{\kappa^{\prime}}(q):=q^{\kappa^{\prime\,2}},\qquad\phi_{\kappa^{\prime}%
}(T_{j}):=T_{j},\qquad\phi_{\kappa^{\prime}}(\omega):=\omega, \qquad
\phi_{\kappa^{\prime}}(x^{\lambda}):=x^{\kappa^{\prime}\lambda}%
\]
for $1\leq i<r$ and $\lambda\in m\mathbb{Z}^{r}$ define a morphism
$\phi_{\kappa^{\prime}}: \mathbb{H}^{(m)}\rightarrow\mathbb{H}^{(n)}$ of
$\mathbb{C}(k)$-algebras. Note that $\phi_{\kappa^{\prime}}(Y^{\lambda
})=Y^{\kappa^{\prime}\lambda}$ for $\lambda\in m\mathbb{Z}^{r}$.

Let $\j_{\kappa^{\prime}}: \mathbb{K}^{(m)}\hookrightarrow\mathbb{K}^{(n)}$ be
the $\mathbb{C}(q)$-homomorphism mapping $q$ to $q^{\kappa^{\prime\,2}}$ and
$g_{j}^{(m)}$ to $g_{\kappa^{\prime}j}^{(n)}$ for all $j\in\mathbb{Z}$. Note
the difference with $\iota_{\kappa^{\prime}}: \mathbb{K}^{(m)}\rightarrow
\mathbb{K}^{(n)}$ (Lemma \ref{iota}), which fixes $q$. The image
$\mathbb{K}^{(n,\kappa^{\prime})}_{\j}$ of the homomorphism $\j_{\kappa
^{\prime}}: \mathbb{K}^{(m)}\rightarrow\mathbb{K}^{(n)}$ is the subfield of
$\mathbb{K}^{(n)}$ obtained by adjoining $q^{\kappa^{\prime\, 2}}$ and
$g_{\kappa^{\prime}j}^{(n)}$ ($1\leq j<\frac{m}{2}$) to $\mathbb{C}(k)$. We
now have the following proposition.

%%%%%%%%%%%%%%%%%%%%%%%%%%%%%%%%%%%%%%%%

\begin{prop}
\label{mn} $\mathbb{K}^{(n,\kappa^{\prime})}_{\j}[x^{\pm\kappa^{\prime}%
}]\subseteq\mathbb{K}^{(n)}[x^{\pm1}]$ is a $(\widehat{\pi}_{n}\circ
\phi_{\kappa^{\prime}},\mathbb{H}^{m})$-submodule. Then
\[
\overline{\j}_{\kappa^{\prime}}\bigl(\sum_{\mu}c_{\mu}x^{\mu}\bigr):=\sum
_{\mu}\j_{\kappa^{\prime}}(c_{\mu})x^{\kappa^{\prime}\mu}\qquad(c_{\mu}%
\in\mathbb{K}^{(m)})
\]
defines an isomorphism
\[
\overline{\j}_{\kappa^{\prime}}: (\mathbb{K}^{(m)}[x^{\pm1}],\widehat{\pi
}^{(m)})\overset{\sim}{\longrightarrow} (\mathbb{K}^{(n,\kappa^{\prime})}_{\j
}[x^{\pm\kappa^{\prime}}],\widehat{\pi}^{(n)}\circ\phi_{\kappa^{\prime}})
\]
of $\mathbb{H}^{(m)}$-modules. In particular, $(\mathbb{K}^{(n,n)}_{\j}[x^{\pm
n}], \widehat{\pi}^{(n)}\circ\phi_{n})$ realizes Cherednik's basic
representation $\widehat{\pi}^{(1)}$ with the role of $q$ replaced by
$q^{n^{2}}$.
\end{prop}

%%%%%%%%%%%%%%%%%%%%%%%%%%%%%%%%%%%%%%%%%%

%%%%%%%%%%%%%%%%%%%%%%%

\subsection{The metaplectic polynomials}\label{sect:metpols}

%%%%%%%%%%%%%%%%%%%%%%%%%%%%%%%

We keep the notations from the previous subsections. In particular,
$(n,\mathbf{Q})$ is the fixed metaplectic data and $\kappa:=\mathbf{Q}%
(\alpha)$ ($\alpha\in\Phi$), leading to the positive integers $\kappa^{\prime
}$ and $m$ by \eqref{defm}. We furthermore fixed a sign $\epsilon\in\{\pm1 \}$
through the definition of the representation parameter $g_{\frac{n}{2}}%
^{(n)}:=\epsilon^{-1}k^{-1}$ if $n$ is even.

The commuting linear operators $\widehat{\pi}^{(n,\kappa)}(Y^{\mu}%
)\in\textup{End}_{\mathbb{K}^{(n,\kappa)}}\bigl(\mathbb{K}^{(n, \kappa
)}[x^{\pm1}]\bigr)$ ($\mu\in m\mathbb{Z}^{r}$) are metaplectic analogs of
Cherednik's $Y $-operators. The following theorem establishes the existence of
a family of Laurent polynomials which are simultaneous eigenfunctions of the
metaplectic $Y$-operators.

For $\mu\in\mathbb{Z}^{r}$ define $\gamma_{\mu}^{(n,\kappa)}\in\textup{Hom}%
_{\mathbb{Z}}\bigl(m\mathbb{Z}^{r},\mathbb{K}^{(n)\times}\bigr)$ by
\[
\gamma_{\mu}^{(n,\kappa)}:=q^{-\mu}\prod_{\alpha\in\Phi^{+}}\bigl(\sigma
^{(n,\kappa)} ((\mu,\alpha^{\vee}))\bigr)^{\frac{\alpha^{\vee}}{m}},
\]
with $\sigma^{(n,\kappa)}: \mathbb{Z}\rightarrow\mathbb{K}^{(n)}$ defined by
\[
\sigma^{(n,\kappa)}(s):=
\begin{cases}
k^{-1}\qquad & \hbox{ if } s\in m\mathbb{Z}_{>0},\\
-kg_{-\kappa s}\qquad & \hbox{ if } s\in\mathbb{Z}\setminus m\mathbb{Z}_{>0}.
\end{cases}
\]
In other words, the value $\bigl(\gamma_{\mu}^{(n,\kappa)}\bigr)^{\lambda}$ of
$\gamma_{\mu}^{(n,\kappa)}$ at $\lambda\in m\mathbb{Z}^{r}$ is
\[
\bigl(\gamma_{\mu}^{(n,\kappa)}\bigr)^{\lambda}= q^{-(\lambda,\mu)}%
\prod_{\alpha\in\Phi^{+}}\bigl(\sigma^{(n,\kappa)} ((\mu,\alpha^{\vee
}))\bigr)^{\frac{(\lambda,\alpha^{\vee})}{m}}%
\]
Note that $\gamma_{\mu}^{(n,\kappa)}$ takes values in $\mathbb{K}^{(n,\kappa
)}$.

%%%%%%%%%%%%%%%%%%%%%%%%%%

\begin{thm}
\label{thm:E-pols} There exists a unique family of Laurent polynomials $\{
E_{\mu}^{(n,\kappa)}(x) \}_{\mu\in\mathbb{Z}^{r}}$ in $\mathbb{K}^{(n,\kappa
)}[x^{\pm1}]$ such that for $\mu\in\mathbb{Z}^{r}$,

\begin{enumerate}
\item[\textbf{(i)}] For all $\lambda\in m \mathbb{Z}^{r}$ we have
\[
\widehat{\pi}^{(n,\kappa)}(Y^{\lambda}) E_{\mu}^{(n,\kappa)}(x)=
\bigl(\gamma_{\mu}^{(n,\kappa)}\bigr)^{\lambda}E_{\mu}^{(n,\kappa)}(x).
\]

\item[\textbf{(ii)}] The coefficient of $x^{\mu}$ in the expansion of $E_{\mu
}^{(n,\kappa)}(x)$ in the monomial basis $\{x^{\nu}\}_{\nu\in\mathbb{Z}^{r}}$,
is one.
\end{enumerate}

We will write $E_{\mu}^{(n)}(x):=E_{\mu}^{(n,1)}(x)$ for $\mu\in\mathbb{Z}%
^{r}$.
\end{thm}

%%%%%%%%%%%%%%%%%%%%%%%%%%%%
The proof of the theorem, including its extension to arbitrary root systems,
will be given in a forthcoming paper. The following proposition is a
consequence of Remark \ref{mnk}\textbf{(ii)} and Proposition \ref{mn}.
%%%%%%%%%%%%%%%%%%%%%%%%%%%%%%%%%%

\begin{prop}
\label{E_transformation1} For all $\mu\in\mathbb{Z}^{r}$,
\begin{equation}
\label{relmm}%
\begin{split}
\overline{\iota}_{\kappa}\bigl(E_{\mu}^{(m)}(x)\bigr) & =E_{\mu}^{(n,\kappa
)}(x),\\
\overline{\j}_{\kappa^{\prime}}\bigl(E_{\mu}^{(m)}(x)\bigr) & = E_{\kappa
^{\prime}\mu}^{(n)}(x).
\end{split}
\end{equation}

\end{prop}

%%%%%%%%%%%%%%%%%%%%%%%%%%%%%%%%%%%%
By the first line of \eqref{relmm}, the metaplectic polynomial $E^{(n,\kappa
)}(x)$ essentially only depends on $m$, $q$, $k$, the representation
parameters $g_{\kappa^{\prime}j}^{(n)}$ ($1\leq j<\frac{m}{2}$) and, if $m$ is
even, on $\epsilon$.

\begin{rema}
\label{classical_case} By Remark \ref{mnk}\textbf{(iii)}, $E_{\mu}^{(1)}(x)$
is the monic nonsymmetric Macdonald polynomial of degree $\mu$ (compared to
the standard conventions on nonsymmetric Macdonald polynomials as in e.g.
\cite{HHL}, $k^{2}$ corresponds to $t$). Furthermore, as a special case of the
second line of \eqref{relmm}, $E_{n\mu}^{(n)}(x)$ realizes the monic
nonsymmetric Macdonald polynomial of degree $\mu\in\mathbb{Z}^{r}$ in the
variables $x_{1}^{n},\ldots,x_{r}^{n}$, with the role of $q$ replaced by
$q^{n^{2}}$.
\end{rema}

%%%%%%%%%%%%%%%%%%%%%%%%%%%%%%%%%%%%%%

%%%%%%%%%%%%%%%%%%%%%%%%%

\subsection{Appendix: table of $GL_{3}$ metaplectic polynomials}

%%%%%%%%%%%%%%%%%%%%%%%%%
We give formulas for $E_{\lambda}^{(m)}(x)$, where $1 \leq m \leq5$ and
$\lambda\in\mathbb{Z}^{3}$ has weight at most $2$. For convenience of
notation, we write $g_{j}$ instead of $g_{j}^{(m)}$. The technique used to
compute these polynomials will be provided in a forthcoming paper. \bigskip

$E_{(0,0,0)}^{(1)}(x) = 1$

$E_{(0,0,0)}^{(2)}(x) = 1$

$E_{(0,0,0)}^{(3)}(x) = 1$

$E_{(0,0,0)}^{(4)}(x) = 1$

$E_{(0,0,0)}^{(5)}(x) = 1$

\bigskip

$E_{(1,0,0)}^{(1)}(x) = x_{1}$

$E_{(1,0,0)}^{(2)}(x) =x_{1}$

$E_{(1,0,0)}^{(3)}(x) =x_{1}$

$E_{(1,0,0)}^{(4)}(x) =x_{1}$

$E_{(1,0,0)}^{(5)}(x) =x_{1}$

\bigskip

$E_{(0,1,0)}^{(1)}(x) = {\frac{\left(  k-1\right)  \left(  k+1\right)  }%
{{k}^{4}q-1}}x_{1}+x_{2}$

$E_{(0,1,0)}^{(2)}(x) = {\frac{\left(  k-1\right)  \left(  k+1\right)
}{k\left(  kq^{2}+\epsilon\right)  }}x_{1}+x_{2}$

$E_{(0,1,0)}^{(3)}(x) ={\frac{\left(  k-1\right)  \left(  k+1\right)  g_{1}%
}{k^{4}g_{1}^{3}q^{3}+1}}x_{1}+x_{2}$

$E_{(0,1,0)}^{(4)}(x) = {\frac{\left(  k-1\right)  \left(  k+1\right)  g_{1}%
}{{k}^{4}g_{1}^{3}q^{4}+1}x_{1}}+x_{2}$

$E_{(0,1,0)}^{(5)}(x) = {\frac{\left(  k-1\right)  \left(  k+1\right)  g_{1}%
}{{k}^{4}g_{1}^{3}q^{5}+1}x_{1}}+x_{2}$

\bigskip

$E_{(0,0,1)}^{(1)}(x) = {\frac{\left(  k-1\right)  \left(  k+1\right)  }%
{q{k}^{2}-1}}x_{1}+{\frac{\left(  k-1\right)  \left(  k+1\right)  }{q{k}%
^{2}-1}}x_{2}+x_{3}$

$E_{(0,0,1)}^{(2)}(x) =-{\frac{\left(  k-1\right)  \left(  k+1\right)
}{k\left(  k+\epsilon q^{2}\right)  }}x_{1}+{\frac{\left(  k-1\right)  \left(
k+1\right)  }{q^{2}+\epsilon k}}x_{2}+x_{3}$

$E_{(0,0,1)}^{(3)}(x) =-{\frac{\left(  k-1\right)  \left(  k+1\right)
g_{1}^{2}}{{k}^{2}g_{1}^{3}q^{3}+1}}x_{1}+{\frac{\left(  k-1\right)  \left(
k+1\right)  g_{1}}{{k}^{2}g_{1}^{3}q^{3}+1}}x_{2}+x_{3}$

$E_{(0,0,1)}^{(4)}(x) = -{\frac{\left(  k-1\right)  \left(  k+1\right)
g_{1}^{2}}{{k}^{2}g_{1}^{3}q^{4}+1}}x_{1}+{\frac{\left(  k-1\right)  \left(
k+1\right)  g_{1}}{{k}^{2}g_{1}^{3}q^{4}+1}}x_{2}+x_{3}$

$E_{(0,0,1)}^{(5)}(x) =-{\frac{\left(  k-1\right)  \left(  k+1\right)
g_{1}^{2}}{{k}^{2}g_{1}^{3}q^{5}+1}}x_{1}+{\frac{\left(  k-1\right)  \left(
k+1\right)  g_{1}}{{k}^{2}g_{1}^{3}q^{5}+1}}x_{2}+x_{3}$

\bigskip

$E_{(0,1,1)}^{(1)}(x) = {\frac{\left(  k-1\right)  \left(  k+1\right)  }%
{q{k}^{2}-1}}x_{1}x_{2}+{\frac{\left(  k-1\right)  \left(  k+1\right)  }%
{q{k}^{2}-1}}x_{3}x_{1}+x_{3}x_{2}$

$E_{(0,1,1)}^{(2)}(x) =-{\frac{\left(  k-1\right)  \left(  k+1\right)
}{k\left(  k+\epsilon q^{2}\right)  }}x_{1}x_{2}+{\frac{\left(  k-1\right)
\left(  k+1\right)  }{q^{2}+\epsilon k}}x_{3}x_{1}+x_{3}x_{2}$

$E_{(0,1,1)}^{(3)}(x) =-{\frac{\left(  k-1\right)  \left(  k+1\right)
g_{1}^{2}}{{k}^{2}g_{1}^{3}q^{3}+1}}x_{1}x_{2}+{\frac{\left(  k-1\right)
\left(  k+1\right)  g_{1}}{{k}^{2}g_{1}^{3}q^{3}+1}}x_{3}x_{1}+x_{3}x_{2}$

$E_{(0,1,1)}^{(4)}(x) =-{\frac{\left(  k-1\right)  \left(  k+1\right)
g_{1}^{2}}{{k}^{2}g_{1}^{3}q^{4}+1}}x_{1}x_{2}+{\frac{\left(  k-1\right)
\left(  k+1\right)  g_{1}}{{k}^{2}g_{1}^{3}q^{4}+1}}x_{3}x_{1}+x_{3}x_{2}$

$E_{(0,1,1)}^{(5)}(x) =-{\frac{\left(  k-1\right)  \left(  k+1\right)
g_{1}^{2}}{{k}^{2}g_{1} ^{3}q^{5}+1}}x_{1}x_{2}+{\frac{\left(  k-1\right)
\left(  k+1\right)  g_{1}}{{k}^{2}g_{1}^{3}q^{5}+1}}x_{3}x_{1}+x_{3}x_{2}$

\bigskip

$E_{(1,0,1)}^{(1)}(x) = {\frac{\left(  k-1\right)  \left(  k+1\right)  }%
{{k}^{4}q-1}}x_{1}x_{2}+x_{3}x_{1}$

$E_{(1,0,1)}^{(2)}(x) = {\frac{\left(  k-1\right)  \left(  k+1\right)
}{k\left(  kq^{2} + \epsilon\right)  }}x_{1}x_{2}+x_{3}x_{1}$

$E_{(1,0,1)}^{(3)}(x) = {\frac{\left(  k-1\right)  \left(  k+1\right)  g_{1}%
}{{k}^{4}g_{1}^{3}q^{3}+1}}x_{1}x_{2}+x_{3}x_{1}$

$E_{(1,0,1)}^{(4)}(x) = {\frac{\left(  k-1\right)  \left(  k+1\right)  g_{1}%
}{{k}^{4}g_{1}^{3}q^{4}+1}}x_{1}x_{2}+x_{3}x_{1}$

$E_{(1,0,1)}^{(5)}(x) = {\frac{\left(  k-1\right)  \left(  k+1\right)  g_{1}%
}{{k}^{4}g_{1}^{3}q^{5}+1}}x_{1}x_{2}+x_{3}x_{1}$

\bigskip

$E_{(1,1,0)}^{(1)}(x) = x_{1}x_{2}$

$E_{(1,1,0)}^{(2)}(x) =x_{1}x_{2}$

$E_{(1,1,0)}^{(3)}(x) = x_{1}x_{2}$

$E_{(1,1,0)}^{(4)}(x) = x_{1}x_{2}$

$E_{(1,1,0)}^{(5)}(x) = x_{1}x_{2}$

\bigskip

$E_{(2,0,0)}^{(1)}(x) = x_{1}^{2}+{\frac{q\left(  k-1\right)  \left(
k+1\right)  }{q{k}^{2}-1}x_{1}x_{2}}+{\frac{q\left(  k-1\right)  \left(
k+1\right)  }{q{k}^{2}-1}x_{3}x_{1}}$

$E_{(2,0,0)}^{(2)}(x) =x_{1}^{2}$

$E_{(2,0,0)}^{(3)}(x) =x_{1}^{2}$

$E_{(2,0,0)}^{(4)}(x) =x_{1}^{2}$

$E_{(2,0,0)}^{(5)}(x) =x_{1}^{2}$

\bigskip

$E_{(0,2,0)}^{(1)}(x) = \frac{\left(  k-1\right)  \left(  k+1\right)
}{\left(  q{k}^{2}-1\right)  \left(  q{k}^{2}+1\right)  }x_{1}^{2}%
+{\frac{\left(  k-1\right)  \left(  k+1\right)  \left(  {k}^{4}{q}^{2}%
+q{k}^{2}-q-1\right)  }{\left(  q{k}^{2}+1\right)  \left(  q{k}^{2}-1\right)
^{2}}}x_{1}x_{2}+{\frac{\left(  k-1\right) ^{2}\left(  k+1\right) ^{2}%
q}{\left(  q{k}^{2}+1\right)  \left(  q{k}^{2}-1\right) ^{2}}}x_{3}x_{1}
\newline\hspace*{2.2cm} + x_{2}^{2}+{\frac{q\left(  k-1\right)  \left(
k+1\right)  }{q{k}^{2}-1}}x_{3}x_{2}$

$E_{(0,2,0)}^{(2)}(x) =\frac{\left(  k-1\right)  \left(  k+1\right)  }{\left(
q^{2}{k}^{2}-1\right)  \left(  q^{2}{k}^{2}+1\right)  }x_{1}^{2}+x_{2}^{2}$

$E_{(0,2,0)}^{(3)}(x) =\frac{\left(  k-1\right)  \left(  k+1\right)  g_{1}%
^{2}}{{k}^{2}g_{1}^{3}+{q}^{6}}x_{1}^{2}+x_{2}^{2}$

$E_{(0,2,0)}^{(4)}(x) =\frac{\left(  k-1\right)  \left(  k+1\right)
}{k\left(  q^{8}k + \epsilon\right)  }x_{1}^{2}+x_{2}^{2}$

$E_{(0,2,0)}^{(5)}(x) =\frac{\left(  k-1\right)  \left(  k+1\right)  g_{2}%
}{{k}^{4}g_{2}^{3}{q}^{10}+1}x_{1}^{2}+x_{2}^{2}$

\bigskip

$E_{(0,0,2)}^{(1)}(x) = \frac{\left(  k-1\right)  \left(  k+1\right)
}{\left(  kq-1\right)  \left(  kq+1\right)  }x_{1}^{2}+{\frac{\left(
q+1\right)  \left(  k-1\right) ^{2}\left(  k+1\right) ^{2}}{\left(
kq-1\right)  \left(  kq+1\right)  \left(  q{k}^{2}-1\right)  }}x_{1}%
x_{2}+{\frac{\left(  q+1\right)  \left(  k-1\right)  \left(  k+1\right)
}{\left(  kq-1\right)  \left(  kq+1\right)  }}x_{3}x_{1}+\frac{\left(
k-1\right)  \left(  k+1\right)  }{\left(  kq-1\right)  \left(  kq+1\right)
}x_{2}^{2} \newline\hspace*{2.2cm} +{\frac{\left(  q+1\right)  \left(
k-1\right)  \left(  k+1\right)  }{\left(  kq-1\right)  \left(  kq+1\right)  }%
}x_{3}x_{2}+x_{3}^{2}$

$E_{(0,0,2)}^{(2)}(x) =\frac{\left(  k-1\right)  \left(  k+1\right)  }{\left(
kq^{2}-1\right)  \left(  kq^{2}+1\right)  }x_{1}^{2}+\frac{\left(  k-1\right)
\left(  k+1\right)  }{\left(  kq^{2}-1\right)  \left(  kq^{2}+1\right)  }%
x_{2}^{2}+x_{3}^{2}$

$E_{(0,0,2)}^{(3)}(x) =-\frac{\left(  k-1\right)  \left(  k+1\right)  g_{1}%
}{{k}^{4}g_{1} ^{3}+{q}^{6}}x_{1}^{2}+\frac{\left(  k-1\right)  \left(
k+1\right)  {k}^{2}g_{1}^{2}}{{k}^{4}g_{1}^{3}+{q}^{6}}x_{2}^{2}+x_{3}^{2}$

$E_{(0,0,2)}^{(4)}(x) =-\frac{\left(  k-1\right)  \left(  k+1\right)
}{k\left(  \epsilon q^{8} + k\right)  }x_{1}^{2}+\frac{\left(  k-1\right)
\left(  k+1\right)  }{q^{8} + \epsilon k}x_{2}^{2}+x_{3}^{2}$

$E_{(0,0,2)}^{(5)}(x) =-\frac{\left(  k-1\right)  \left(  k+1\right)
g_{2}^{2}}{{k}^{2}g_{2}^{3}{q}^{10}+1}x_{1}^{2}+\frac{\left(  k-1\right)
\left(  k+1\right)  g_{2}}{{k}^{2}g_{2}^{3}{q}^{10}+1}x_{2}^{2}+x_{3}^{2}$

\begin{rema}
We mention a few general properties of the metaplectic polynomials which can
be observed in the table above.

\begin{enumerate}
\item The following are monic nonsymmetric Macdonald polynomials:
$E^{(1)}_{\lambda}(x)$ (for any $\lambda$), $E^{(2)}_{(2,0,0)}(x)$,
$E^{(2)}_{(0,2,0)}(x)$, and $E^{(2)}_{(0,0,2)}(x)$ (see Remark
\ref{classical_case}). In particular, the formulas given above for these
polynomials match the ones provided in the appendix of \cite{HHL} (with
$k^{2}$ replaced by $t$).

%For any $m \in \mathbb{Z}_{\geq 1}$ and $\lambda \in \mathbb{Z}^{3}$, $E^{(m)}_{m \lambda}$ is the monic nonsymmetric Macdonald polynomial indexed by $\lambda$ with variables $x_i^{m}$ and parameter $q^{m^2}$.

%For example, $E^{(2)}_{(0,2,0)}(x)$ is the Macdonald polynomial $E_{(0,1,0)}(x_1^{2}, x_2^{2}; q^4, k)$.  x

\item More generally, for any $a \in\mathbb{Z}_{\geq1}$, the metaplectic
polynomial $E^{(am)}_{a \lambda}(x)$ may be obtained from $E^{(m)}_{\lambda
}(x)$ via the substitutions $x_{i} \to x_{i}^{a}$, $q \to q^{a^{2}} $ and
$g_{j}^{(m)} \to g_{aj}^{(am)}$. This follows directly from Proposition
\ref{E_transformation1} with $\kappa^{\prime}= a$. We list the pairs
$\big(E^{(a)}_{\lambda}(x), E^{(am)}_{m \lambda}(x)\big)$ from the table with
$a \neq1$ for which this applies:

\begin{enumerate}
\item $\big(E^{(2)}_{(0,0,0)}(x), E^{(4)}_{(0,0,0)}(x)\big)$

\item $\big(E^{(2)}_{(1,0,0)}(x), E^{(4)}_{(2,0,0)}(x)\big)$

\item $\big(E^{(2)}_{(0,1,0)}(x), E^{(4)}_{(0,2,0)}(x)\big)$

\item $\big(E^{(2)}_{(0,0,1)}(x), E^{(4)}_{(0,0,2)}(x)\big)$.
\end{enumerate}

%For example, $E^{(4)}_{(0,2,0)}(x)$ is obtained from $E^{(2)}_{(0,1,0)}(x)$ by substituting $x_i^{2}$ for $x_i$ and $q^{4}$ for $q$.

\item For $GL_{3}$, we have
\[
C_{+}^{m} = C^{m} \cap P_{+} = \{ \lambda\in\mathbb{Z}^{3}: \lambda_{1}
\geq\lambda_{2} \geq\lambda_{3}, \lambda_{1} - \lambda_{3} \leq m \}
\]
(see \eqref{C} for the definition of $C^{m}$). If $\lambda\in C_{+}^{m}$, the
metaplectic polynomial $E^{(m)}_{\lambda}(x)$ is equal to the monomial
$x^{\lambda}$. This will be proved in the followup paper in the context of
arbitrary root systems. Note that this result applies to the following
polynomials listed above: $E^{(m)}_{(0,0,0)}(x)$, $E^{(m)}_{(1,0,0)}(x)$,
$E^{(m)}_{(1,1,0)}(x)$ (any $m \in\mathbb{Z}_{\geq1}$) and $E^{(m)}%
_{(2,0,0)}(x)$ (any $m \in\mathbb{Z}_{\geq2}$). Note that for $\lambda\in m
\mathbb{Z}^{3} \cap C_{+}^{m}$, this recovers the well-known fact that the
nonsymmetric Macdonald polynomial corresponding to the miniscule weight
$\lambda$ is a monomial.
\end{enumerate}
\end{rema}

%%%%%%%%%%%%%%%%%%%%%%%%%%%%%%%%%%%%%%%%%%%%%%

\end{document}